\documentclass{amsart}

\usepackage{amsrefs}
\usepackage{thesis_header}
\usepackage{enumitem}
\usepackage{tikz-cd}
\newcommand{\noopsort}[1]{}

\DefineSimpleKey{bib}{myurl}
\newcommand\myurl[1]{\url{#1}}
\newcommand\phdize[1]{Ph.D. thesis, #1}

\BibSpec{electronic}{
  +{}{\PrintAuthors} {author}
  +{,}{ \textit} {title}
  +{}{ \parenthesize} {date}
  +{,}{ \myurl} {url}
  +{,}{ } {note}
  +{.}{ } {transition}
}
\BibSpec{manual}{
  +{}{\PrintAuthors} {author}
  +{,}{ \textit} {title}
  +{}{ \parenthesize} {date}
  +{,}{ \myurl} {url}
  +{,}{ } {note}
  +{.}{ } {transition}
}
\BibSpec{thesis}{
  +{}{\PrintAuthors} {author}
  +{,}{ \textit} {title}
  +{.}{ \phdize } {school}
  +{,}{ } {year}
  +{.}{ } {transition}
}
\BibSpecAlias{inbook}{inproceedings}

\title{Constructing local $L$-packets for tame unitary groups}
\author{David Roe}
\email{roed.math@gmail.com}
\address{Department of Mathematics, University of Calgary, Calgary, AB T2N 1N4, Canada}
\subjclass[2010]{22E50, 11S37}
\keywords{Langlands, unitary groups}
\thanks{The author was supported by the Pacific Institute for the Mathematical Sciences while adapting this manuscript from his thesis}
\date{\today}

\begin{document}
\newtheorem{theorem}{Theorem}[section]
\newtheorem{proposition}[theorem]{Proposition}
\newtheorem{lemma}[theorem]{Lemma}
\newtheorem{corollary}[theorem]{Corollary}
\newtheorem{conjecture}[theorem]{Conjecture}
\newtheorem{construction}[theorem]{Construction}
\theoremstyle{definition}
\newtheorem{definition}[theorem]{Definition}
\newtheorem{remark}[theorem]{Remark}
\numberwithin{equation}{section}

\begin{abstract}
  \addcontentsline{toc}{section}{Abstract}
We generalize the work of DeBacker and Reeder \cite{reeder-debacker:09a} to the case of unitary groups split by a tame extension.  The approach is broadly similar and the restrictions on the parameter the same, but many of the details of the arguments differ.

Let $\G$ be a unitary group defined over a local field $\K$ and splitting over a tame extension $E/\K$.  Given a Langlands parameter $\varphi \colon \Weil_\K \rightarrow {^L\G}$ that is tame, discrete and regular, we give a natural construction of an $L$-packet $\Pi_\varphi$ associated to $\varphi$, consisting of representations of pure inner forms of $\G(\K)$ and parameterized by the characters of the finite abelian group $A_\varphi = \Z_{\hat{\G}}(\varphi)$.

\end{abstract}
\maketitle

\section{Introduction}

The local Langlands correspondence has enjoyed great success in recent years, with proofs for $\GL_n$ \cites{harris-taylor:01a,henniart:00a} and other classical groups \cite{arthur:EndoscopicClassification}.  However, these proofs do not give an explicit construction of the $L$-packet associated to a particular Langlands parameter.  A different approach, initiated by DeBacker-Reeder \cite{reeder-debacker:09a}, fills this gap at the cost of restricting the class of Langlands parameters appearing on one side of the correspondence.  In this paper we extend the constructions of DeBacker-Reeder to tamely ramified unitary groups.

\subsection*{The DeBacker-Reeder case}

Let $\K$ be a finite extension of $\Qp$ and suppose that $\G$ is a quasi-split, connected, reductive group defined over $\K$ and splitting over an \emph{unramified} extension $E/\K$.  Write $\Weil_\K$ for the Weil group of $K$, $\hat{\G}$ for the connected, reductive group over $\CC$ with root datum dual to that of $\G$ and ${}^L\G = \hat{\G} \rtimes \Gal(E/\K)$ for a Langlands dual group.  DeBacker and Reeder consider Langlands parameters $\varphi : \Weil_\K \to {}^L\G$ that are
\begin{enumerate}
\item \emph{tame}: $\varphi$ factors through the quotient of $\Weil_\K$ by wild inertia,
\item \emph{discrete}: the centralizer of $\varphi$ in $\hat{\G}$ is finite modulo the center of ${}^L\G$, and
\item \emph{regular}: the image of inertia is generated by a semisimple element of ${}^L\G$ whose centralizer in $\hat{G}$ is a maximal torus $\hat{\Ss}$.
\end{enumerate}
We summarize their construction here to highlight the similarities and differences with our version.

Suppose that $\lambda \in X^*(\hat{\Ss})$.  Given $\G, \varphi$ and $\lambda$ they construct pairs $(\pi_\lambda, \Fr_\lambda)$, where $\Fr_\lambda$ is a twist of Frobenius, $\G^{\Fr_\lambda}$ are the $\K$-points of the pure inner form of $\G$ determined by $\Fr_\lambda$, and $\pi_\lambda$ is a representation of $\G^{\Fr_\lambda}$.  They then define a notion of equivalence of such pairs and prove that the equivalence class of $(\pi_\lambda, \Fr_\lambda)$ depends only on the class of $\lambda$ in a finite quotient of $X^*(\hat{\Ss})$ isomorphic to $\Irr(A_\varphi)$.

The first step in the construction of $\pi_\lambda$ is the construction of a point $x_\lambda$ in the Bruhat-Tits building $\BB(\G)$ as the unique fixed point of a specific automorphism of the apartment $X^*(\hat{\Ss}) \otimes \RR$.  This point in the building determines a maximal compact subgroup $\G_\lambda$ used in the construction of $\pi_\lambda$.  From $x_\lambda$ they also obtain an unramified anisotropic maximal torus $\Ss_\lambda$ as a particular twist of a fixed maximal torus $\Ss \subset \G$.  

The image of $\varphi$ is contained within the normalizer of $\hat{\Ss}_{\lambda}$.  If the image were in fact a semidirect product, then the local Langlands correspondence for tori would give a character of $\Ss_{\lambda}(\K)$.  In general there is no such semidirect product decomposition of the image, but DeBacker and Reeder are able to modify $\varphi$ in a canonical way to obtain a new parameter whose image can be expressed as a semidirect product and thus defines a character on $\Ss_\lambda(\K)$.  They can then use Deligne-Lusztig theory to define a representation of the parahoric subgroup $\G_\lambda$, which compactly induces to the desired supercuspidal representation of $\G^{\Fr_\lambda}$. 

One benefit of the DeBacker-Reeder approach is that it explicitly constructs the representations in an $L$-packet from the data of a Langlands parameter.  It also works for a broad class of groups $\G$.

\subsection*{Expanding upon DeBacker-Reeder}

In this paper we will expand and modify the methods of DeBacker and Reeder to remove their condition that $\G$ splits over an unramified extension: this generalization constitutes our main result.  We will continue to assume that $\varphi$ factors through the tame Weil group, and this requirement forces $\G$ to split over a tamely ramified extension of $\K$.  While our methods should apply to tame, discrete, regular parameters for arbitrary $\G$, we focus here on the case of unitary groups as a concrete example.

Suppose that $\K$ is a finite extension of $\Qp$ with $p \ne 2$, $E / \K$ is a ramified quadratic extension of $\K$, $V$ is a Hermitian space over $E$, and $\G$ is the unitary group associated to $V$.  We require that $\varphi$ is tame, discrete and regular as in DeBacker-Reeder.

Any unitary group has a pure inner form that is quasi-split.  A Langlands parameter for $\G$ will yield representations of all pure inner forms of $\G$, so we may assume that $\G$ itself is quasi-split.  Let $\Ss$ be the centralizer of a maximal $\K$-split torus in $\G$; since $\G$ is quasi-split $\Ss$ will be a maximal torus defined over $\K$, uniquely determined up to conjugation within $\G(\K)$.  We use $\Ss$ in the construction of the Langlands dual group $\hat{\G}$, and thus $\hat{\G}$ comes equipped with a maximal torus $\hat{\Ss}$ over $\CC$ dual to $\Ss$.

The first few sections of the paper describe the structures on unitary groups needed to construct $L$-packets.  
In section \ref{sec:HermUnit}, we review Hermitian spaces and unitary groups, and we give a classification over $p$-adic fields.
We continue in Section \ref{sec:tori} by describing tori in unitary groups, culminating in a concrete description of the relatively unramified, anisotropic tori.
These tori also hold the focus of Section \ref{sec:embedding_tori}, where we describe how they embed in different unitary groups.
In Section \ref{sec:buildings} we describe the Bruhat-Tits building of a unitary group and discuss the reductions modulo $\pi$ of the anisotropic tori and parahorics.
We also explain why each relatively unramified anisotropic torus embeds in a unique maximal compact subgroup of $\G$.

With this background in hand, the second part of the paper describes the construction of an $L$-packet from a tame, discrete, regular Langlands parameter $\varphi$.
In Section \ref{sec:ua_tori:construction} we show that any tame Langlands parameter can be conjugated to particular form, and use the image of Frobenius to construct an anisotropic, relatively unramified torus $\T$.  We would like to use the local Langland correspondence for tori to define a character of $\T(\K)$, but are thwarted by the fact that the normalizer of $\hat{\T}$ does not split as a semidirect product.  In Section \ref{sec:character:typeL} we develop groups of type $L$ to sidestep this difficulty.  Using groups of type $L$, we define in Section \ref{sec:char_from_LP} a depth-zero character of $\T(\K)^0$.  We also show that this character is regular, a property used in demonstrating that the associated Deligne-Lusztig representation is irreducible.  Finally, Section \ref{sec:induction:rep_const} puts the pieces together and defines the supercuspidal representations that go into an $L$-packet.  For odd unitary groups the induction process yields a reducible representation, and we describe how to pick out one of the two constituents using a recipe for the central character.

Since DeBacker-Reeder have described $L$-packets for unramified unitary groups, we focus here on the case that $\G = \U(V)$ splits over a ramified quadratic extension $\E/\K$.  However, many of the results in Sections \ref{sec:ua_tori:construction}-\ref{sec:char_from_LP} work for any tamely ramified, connected, reductive group.  The primary exception is that we use the fact that $q \equiv 1 \pmod{\lvert \Gal(\E/\K) \rvert}$ in Construction \ref{con:ua_tori:T_construction}.

\subsection*{Acknowledgements}

A version of the results in this article were completed as part of my Ph.D. thesis \cite{roe:11a}.  I would like to thank my advisor, Benedict Gross, for his years of support and enthusiasm, and for encouraging me to transform my dissertation into this paper.

I greatly appreciate those mathematicians who contributed discussions and feedback to this work.  Jiu-Kang Yu's course on Bruhat-Tits buildings proved tremendously useful, and I thank him for his patience in explaining $p$-adic representation theory.  Mark Reeder's careful reading caught a host of errors and  simplified the arguments for Lemma \ref{lem:ua_tori:cent_ttau} and Proposition \ref{prop:ua_tori:Frob_in_N}; I am grateful for his expertise.  Finally, I would like to thank Andrew Fiori for proofreading an early draft and providing valuable feedback on which sections needed clarification.

I am grateful for the NDSEG and NSF fellowships that provided financial support during my five years at Harvard, and for the PIMS postdoctoral fellowship that supported me during the writing of this paper.

\setcounter{tocdepth}{1}
\tableofcontents

%

\section{Hermitian spaces and unitary groups} \label{sec:HermUnit}

In this section we give a review of Hermitian spaces and unitary groups over $p$-adic fields.

Let $K$ be a finite extension of $\Qp$ and fix an algebraic closure $\barK$.  All finite extensions of $\K$ are considered to be subfields of $\bar{\K}$.  We will write $\OK$ for the ring of integers in $\K$, $\k$ for the residue field and $\GalK = \Gal(\barK/\K)$ for the absolute Galois group.  Set $\Ku$ as the maximal unramified subfield of $\barK$, $\I = \Gal(\barK/\Ku)$ as the inertia subgroup of $\GalK$ and $\GalKuK = \Gal(\Ku/\K)$ as the quotient.  Let $\E/\K$ be a separable, quadratic extension of $\K$ and let $\tau$ be the nontrivial element of $\GalEK$.
Suppose that $V$ is a free $\E$-module of rank $n$.

\begin{definition} $ $
\begin{enumerate}
\item A \emph{Hermitian form} on $V$ is a function $\phi \colon V \times V \rightarrow \E$ satisfying
\begin{enumerate}
\item for all $v_1,v_2, v' \in V$ and $a, b \in E$, $\phi(av_1 + bv_2, v') = a\phi(v_1, v')+b\phi(v_2, v')$,
\item for all $v, v' \in V$, $\phi(v, v') = \tau \phi(v', v)$.
\end{enumerate}
\item A Hermitian form is \emph{nondegenerate} if $\phi(v, v') = 0$ for all $v' \in V$ implies $v=0$.
\item A \emph{Hermitian space} is a free $\E$-module with a nondegenerate Hermitian form.
\item The \emph{unitary group} $U(V) \subset \Res_{E/K} \GL(V)$ is the subgroup preserving $\phi$.
\end{enumerate}
\end{definition}

We may associate to $V$ a fundamental invariant.  Suppose $v_1, \ldots, v_n$ is a basis for $V$.

\begin{definition} \label{def:disc}
The \emph{discriminant} $dV$ of $V$ is the determinant of the matrix $\left[\phi(v_i, v_j)\right]$, well defined as an element of $\K^\times / \Nm_{E/\K} E^\times$.
\end{definition}

The discriminant behaves multiplicatively with respect to the orthogonal sum $V \oplus W$ of two Hermitian spaces $V$ and $W$:
\begin{equation} \label{eq:background:disc_sum}
d(V \oplus W) = dV \cdot dW.
\end{equation}

\subsection*{\texorpdfstring{Hermitian spaces over $p$-adic fields}{Hermitian spaces over p-adic fields}}

We assume now that $K$ is a $p$-adic field.  Classification of Hermitian spaces in this case relies on the following result:

\begin{proposition}[{c.f. \cite{gerstein:BasicQuadraticForms}*{Thm 2.67, \S 4.4}}] \label{prop:disc_class}
Two Hermitian spaces $V$ and $W$ associated to $E / \K$ are isometric if and only if they have the same dimension and same discriminant.  
\end{proposition}

Thus there are precisely two isometry classes of Hermitian space in each dimension since $\K^\times / \Nm_{E/\K} E^\times$ has order 2 by local class field theory.  We now give concrete descriptions of these Hermitian spaces.

\begin{description}[style=nextline]
\item[$\bullet\ \ V$ has dimension $1$ over $E$]
Let $v \in V$ be nonzero.  The discriminant is given by the class of 
\[
\phi(v, v) \in \K^\times / \Nm_{E/\K} E^\times.
\]
An element $\alpha \in \GL(V) \cong E^\times$ will preserve $\phi$ if and only if $\Nm_{E/\K}(\alpha) = 1$, regardless of the value of $\phi(v, v)$.  Thus the unitary groups associated to the two Hermitian spaces of dimension one are isomorphic.  We refer to this group as $\U(E / \K)$, or $\U_1$ if the extension $E / \K$ is fixed.
\item[$\bullet\ \ V$ has dimension $2$ over $E$ and has an isotropic vector]
Suppose that there is an isotropic vector in $V$, namely some $v \in V$ with $v \ne 0$ and $\phi(v,v) = 0$.  Since $\phi$ is nondegenerate, for $w \in V$ not a multiple of $v$ we have $\phi(v, w) \ne 0$.  By adjusting $\alpha \in \K$, we can force
\[
\phi(w + \alpha v, w + \alpha v) = \phi(w,w) + \Tr_{E/\K} \left(\alpha \phi(v, w)\right)
\]
to be zero, since $\Tr_{E/\K}$ is surjective.  By rescaling the resulting vector, we can find a $w \in V$ with $\phi(w, w) = 0$ and $\phi(v,w) = \phi(w,v) = 1$.  Therefore, any two Hermitian spaces with an isotropic vector are isometric.  We will call this space the \emph{hyperbolic plane} associated to $E/\K$ and denote it by $\Hyp$.  The discriminant of the hyperbolic plane is clearly $d\Hyp = -1$.  
\item[$\bullet\ \ V$ has dimension $2$ over $E$ and has no isotropic vector]
The other isometry class of two-dimensional Hermitian spaces has no isotropic vector.  One method for constructing it takes advantage of the fact that we know that its discriminant must be different from $-1$ modulo $\Nm_{E/\K} E^\times$.  For any $\alpha, \beta \in \K^\times$ with $\alpha\beta \not\equiv -1 \pmod{\Nm_{E/\K} E^\times}$, and any basis $v, w$ of $V$ we can define $\phi$ by
\begin{align} \label{eq:background:nonhyper}
\phi(v, v) = \alpha, \nonumber \\
\phi(v, w) = 0, \\
\phi(w, w) = \beta. \nonumber
\end{align}

Any two such spaces are isometric, and we will refer to this other isometry class of two-dimensional Hermitian space as the \emph{anisotropic plane} and denote it by $\Ani$.

One can also start with a quaternion algebra $\Ani$ containing $\E$ and put the structure of a Hermitian space on it: see Gross \cite{gross:04a}*{\S 5} and Springer \cite{springer:LinearAlgebraicGroups}*{\S 17.1.4} for details.  If this quaternion algebra is split we get an isomorphism with $\Hyp$; the non-split case yields the anisotropic planes.  In either case, the associated group of unitary transformations can be identified with $\Ani^\times$.

Note that the unitary groups $\U(\Hyp)$ and $\U(\Ani)$ are not isomorphic: $\U(\Hyp)$ contains a $\K$-split torus of dimension $1$, while $\U(\Ani)$ does not contain a nontrivial $\K$-split torus.
\end{description}

We can express any higher dimensional Hermitian space as an orthogonal sum of these one and two-dimensional spaces:
\begin{proposition} \label{prop:background:Uclass}
Suppose $V$ is a Hermitian space of dimension $n$.  Then
\begin{enumerate}
\item if $n = 2m$, then either
\begin{align*}
V &\cong \Hyp^m & dV &\equiv (-1)^m \\
\intertext{or}
V &\cong \Hyp^{m-1} \oplus \Ani & dV &\nequiv (-1)^m.
\end{align*}
We will call such unitary groups \emph{even}.  We have $U(\Hyp^m) \not\cong U(\Hyp^{m-1} \oplus \Ani)$.
\item If $n=2m+1$, then 
\begin{align*}
V &\cong \Hyp^m \oplus L & dV &\equiv (-1)^m dL
\end{align*}
for some one-dimensional Hermitian space $L$.  We will call such unitary groups \emph{odd}.  We have $\U(\Hyp^m \oplus L) \cong \U(\Hyp^m \oplus L')$ for any Hermitian lines $L$ and $L'$.
\end{enumerate}
\end{proposition}
\begin{proof}
We obtain each possible discriminant and thus each isometry class of Hermitian space.  The inequality between even unitary groups follows from the difference in the dimension of the maximal $K$-split torus in the two cases.  For odd Hermitian spaces, scaling the Hermitian form by an element of $\K^\times - \Nm_{E/\K} E^\times$ changes the discriminant but leaves the notion of unitary transformation invariant.  So the two different isometry classes of Hermitian space in odd dimensions yield isomorphic unitary groups.
\end{proof}

It will be useful to specify a basis for $V$ in each case. 
\begin{itemize}
\item When $V \cong \Hyp^m$, let $\{v_i, v_{-i}\}$ be the standard basis for the $i\th$ hyperbolic plane.
\item When $V \cong \Hyp^m \oplus L$, let $\{v_i, v_{-i}\}$ be the standard basis for the $i\th$ hyperbolic plane, and $\{v_0\}$ a basis for $L$.  We require $\phi(v_0, v_0) \in \OO_\K^\times$.
\item When $V \cong \Hyp^{m-1} \oplus \Ani$, we let $\{v_i, v_{-i}\}$ be the standard basis for the $i\th$ hyperbolic plane, and then choose two orthogonal vectors $v_0, v_0' \in \Ani$.  We can normalize the choice of $v_0$ and $v_0'$ by imposing the same conditions as on $v_0$ above.
\end{itemize}

\section{Tori in unitary groups} \label{sec:tori}

In this section we describe a certain class of tori in unitary groups that will play a central role in the construction of $L$-packets: the relatively unramified, anisotropic tori.

\subsection*{Maximal split tori}

A \emph{maximal $\K$-split torus} in an algebraic group $\G$ is a subtorus that is maximal among those that are split over $\K$; any two such tori are conjugate over $\K$ \cite{springer:LinearAlgebraicGroups}*{Thm. 15.2.6}.  Moreover, we can find a maximal $\K$-torus containing any given maximal $\K$-split torus $\Aa \subset \G$, since $\Z_{\G}(\Aa)$ contains a maximal torus defined over $\K$ \cite{springer:LinearAlgebraicGroups}*{Thm. 13.3.6}.

We say that $\G$ is \emph{quasi-split} if one of the following conditions hold.
\begin{proposition}[{c.f. \cite{springer:LinearAlgebraicGroups}*{Prop. 16.2.2}}] \label{prop:background:qs}
The following are equivalent:
\begin{enumerate}
\item the centralizer $\Z_{\G}(\Aa)$ of any maximal $\K$-split torus $\Aa$ is a maximal torus,
\item \label{proppart:background:qs_c} there is a Borel subgroup of $\G$ that is defined over $\K$.  
\end{enumerate}
\end{proposition}
Note that since $\Aa$ is determined up to conjugacy, the first criterion allows us to pick out a $\G(\K)$ conjugacy class of maximal tori, which we will refer to as the \emph{quasi-split maximal tori} in $\G$.

In order to understand these tori for unitary groups over $p$-adic fields, we begin with the hyperbolic plane $\Hyp$.  Let $\Ss'$ be the maximal torus in $\GL(\Hyp)$ consisting of those matrices with $v_{-1}$ and $v_1$ as eigenvectors, and define $\Ss \subset \U(\Hyp)$ as the intersection of $\Res_{E/\K} \Ss'$ with $\U(\Hyp) \subset \Res_{E/\K} \GL(\Hyp)$.  Then $\Ss$ is isomorphic to $\Res_{E/\K} \Gm$, with $\K$-points consisting of those matrices scaling $v_1$ by $\alpha \in E^\times$ and $v_{-1}$ by $\tau(\alpha)^{-1}$.  We now define $\Aa$ as the maximal $\K$-split subtorus of $\Ss$, which is also a maximal $\K$-split subtorus in $\U(V)$.  The $\K$-points of $\Aa$ consist of those matrices scaling $v_1$ by $\alpha \in \K^\times$ and $v_{-1}$ by $\alpha^{-1}$.  

We can choose a basis $\{\chi_1, \chi_{-1}\}$ of $X^*(\Ss)$ so that $\tau \in \GalEK$ acts by $\tau(\chi_1) = -\chi_{-1}$ \cite{springer:LinearAlgebraicGroups}*{Prop. 11.4.22}.  Note that $\chi_1$ and $\chi_{-1}$ are not the characters that pick out the eigenvalues of $v_1$ and $v_{-1}$.  Rather, restriction induces an orthogonal projection $X^*(\Ss) \rightarrow X^*(\Aa)$ with kernel spanned by $\chi_1 + \chi_{-1}$ and leaving $\chi_1 - \chi_{-1}$ fixed.  We identify $X^*(\Aa)$ with the span of $\chi_1 - \chi_{-1}$.  The complementary subspace of $X^*(\Ss)$ spanned by $\chi_1 + \chi_{-1}$ corresponds to $\Aa' \subset \Ss$ isomorphic to $\U_1$.

Using these tori in $\Hyp$, we may describe the maximal $K$-split tori and their centralizers for all of the unitary groups listed in Proposition \ref{prop:background:Uclass}.
\begin{itemize}
\item The quasi-split torus $\Ss$ in $\U(\Hyp \oplus \cdots \oplus \Hyp)$ merely uses more indices.  We can write down a basis $\{\chi_{-m}, \ldots, \chi_{-1}, \chi_1, \ldots, \chi_m\}$ of $X^*(\Ss)$ so that $\GalK$ acts through its quotient $\GalEK$, with $\tau$ mapping $\chi_i$ to $-\chi_{-i}$.  One can identify $\Aa$ as the subtorus corresponding to the span of $\{\chi_i - \chi_{-i}\}_{i=1}^m$, and we have
\begin{align*}
\Aa &\cong (\Gm)^m, \\
\Ss &\cong (\Res_{E/\K} \Gm)^m.
\end{align*}
\item In $\U(\Hyp \oplus \cdots \oplus \Hyp \oplus L)$ we add $\chi_0$ to the basis for $X^*(\Ss)$; $\GalK$ still acts through $\GalEK$ with $\tau$ mapping $\chi_i$ to $-\chi_{-i}$.  We have
\begin{align*}
\Aa &\cong (\Gm)^m, \\
\Ss &\cong (\Res_{E/\K} \Gm)^m \times \U_1(E/K).
\end{align*}
\item Both of the previous cases are quasi-split, with $\Ss$ equal to the centralizer of $\Aa$.  The maximal $\K$-split torus in $\U(\Hyp \oplus \cdots \oplus \Hyp \oplus \Ani)$ is just the one included from $\U(\Hyp \oplus \cdots \oplus \Hyp)$, but its centralizer is $(\Res_{E/K} \Gm)^m \times \Ani$.  By Proposition \ref{prop:background:qs} $\U(\Hyp \oplus \cdots \oplus \Hyp \oplus \Ani)$ is not quasi-split.
\end{itemize}

\subsection*{Weyl groups}

In studying anisotropic tori in unitary groups, it will be important to understand the Weyl group of $\Ss$ in the quasi-split case.  Let $\Nn$ be the normalizer of $\Ss$ in $\G = \U(V)$, and $\W = \Nn / \Ss$ the Weyl group of $\Ss$.  Since $\U(V)$ is an inner form of $\GL(V)$, $\W$ is isomorphic to the symmetric group $\Perm_n$, generated by reflections $\omega_{i,j}$ in the roots $\chi_i - \chi_j$.  Define $\eta \in \W$ as a product of commuting reflections:
\[
\eta = \prod_{i=1}^m \omega_{i,-i}.
\]
The action of $\GalK$ on $\W$ is determined through the actions of $\GalK$ and $\W$ on $X^*(\Ss)$:
\[
(\gamma.\sigma)(\chi) = \gamma(\sigma(\gamma^{-1} \chi)) \mbox{ for } \sigma \in \W \mbox{ and } \gamma \in \GalK.
\]
A computation shows that in fact this action is inner:
\[
\tau . \sigma = \eta\sigma\eta.
\]

\begin{proposition} \label{prop:ua_tori:WI_for_2An}
The rational Weyl group $\W^\GalK = \Z_{\W}(\eta)$ is isomorphic to $(\ZZ/2\ZZ)^m \rtimes \Perm_m$.
\end{proposition}

\begin{proof}
The transpositions $\omega_{i,-i}$ and the elements $\omega_{i,j}\omega_{-i,-j}$ for $i,j > 0$ visibly commute with $\eta$; these generate a subgroup of $\W$ of the desired form.  The cardinality of $\Z_{\W}(\eta)$ is $2^m m!$ by the orbit-stabilizer theorem and the computation that there are $\frac{n!}{2^m m!}$ permutations conjugate to $\eta$.
\end{proof}

\begin{remark}
Reeder \cite{reeder:11a} includes a discussion of centralizers in Weyl groups.  He focuses on the Weyl group of $\Eop_8$, but his techniques are applicable to a general Weyl group.
\end{remark}

We will need a description of the conjugacy classes in $\W^\GalK$.  Since $\W^\GalK \cong \Wop(\Bop_m)$, we can turn to Carter \cite{carter:72a}*{pp. 25-26}.

The rational Weyl group $\W^\GalK$ acts on the set $\{ \chi_{\pm i}\}_{i=1}^m$, and one can decompose any element $w \in \W^\GalK$ into cycles on these vectors.  Such cycles take the form
\[
\chi_{i_1} \mapsto \chi_{\pm i_2} \mapsto \chi_{\pm i_3} \mapsto \cdots \mapsto \chi_{\pm i_r} \mapsto \chi_{\pm i_1}.
\]
\begin{definition} \label{def:ua_tori:positive_cycle}
We say that such a cycle $(i_1, \ldots, i_r)$ is \emph{positive} if $w^r(\chi_{i_1}) = \chi_{i_1}$ and \emph{negative} if $w^r(\chi_{i_1}) = \chi_{-i_1}$.  We call $r$ the \emph{length} of the cycle.  The collection of lengths and signs of the cycles of $w$ is called the \emph{signed cycle type} of $w$.
\end{definition}

Given such a signed cycle type, we can define a pair of partitions $\mu$ and $\nu$ by setting $\mu$ to be the collection of lengths of positive cycles and $\nu$ to be the collection of lengths of negative cycles.  For example, when $n=6$ the element $\omega_{1,-1}\omega_{2,-2}$ yields $\mu = \{1\}$ and $\nu = \{1,1\}$.

\begin{proposition}[{c.f. \cite{carter:72a}*{Prop. 24}}] \label{prop:ua_tori:conj_in_Bn} \mbox{}
\begin{enumerate}
\item A signed cycle type occurs for some element of $\W^\GalK$ if and only if $\lvert \mu \rvert + \lvert \nu \rvert = m$.
\item The conjugacy classes of $\W^\GalK$ are in bijection with the possible signed cycle types.  
\end{enumerate}
\end{proposition}

\subsection*{Relatively unramified tori}

Assume now that $\G$ is quasi-split.  Following Reeder \cite{reeder:11a}*{\S 6}, we describe tori in $\G$ as Galois twists of the quasi-split torus $\Ss$.

We say that two tori $\Ss_1$ and $\Ss_2$ are \emph{rationally conjugate} if there is an element of $\G(\K)$ conjugating $\Ss_1(\K)$ to $\Ss_2(\K)$, and \emph{stably conjugate} if there is an element of $\G(\bar{\K})$ conjugating $\Ss_1(\K)$ to $\Ss_2(\K)$.  These notions partition the $\K$-tori in $\G$ into stable conjugacy classes, and each stable conjugacy class into rational conjugacy classes.  We have maps 
\[
\Hh^1(\K, \Nn) \rightarrow \Hh^1(\K, \W),
\]
induced by the projection $\Nn \rightarrow \W$, and
\[
\Hh^1(\K, \Nn) \rightarrow \Hh^1(\K, \G),
\]
induced by the inclusion $\Nn \rightarrow \G$.  These cohomology groups give us a parameterization of the $\K$-conjugacy classes of maximal tori in $\G$ and its pure inner forms.

\begin{proposition}[{c.f. \cite{reeder:11a}*{Prop 6.1} and \cite{serre:GaloisCohomology}*{Cor. 2 of Prop. I.36}}]
Let $\G$ be a quasi-split group over $\K$, $\xi \in \Zcycle^1(\K, \G)$ be a cocycle and $\G_\xi$ be the twist of $\G$ corresponding to $\xi$.
\begin{enumerate}
\item The rational classes of maximal tori in $\G_\xi$ are in bijection with the set $R_\xi$ of cohomology classes in $\Hh^1(\K, \Nn)$ mapping to the class of $\xi$ in $\Hh^1(\K, \G)$.  In particular, the rational classes of maximal tori in $\G$ are in bijection with the kernel of $\Hh^1(\K, \Nn) \rightarrow \Hh^1(\K, \G)$. 
\item The stable classes of maximal tori in $\G_\xi$ are in bijection with the image of $R_\xi$ in $\Hh^1(\K, \W)$.
\item The stable classes of maximal tori in $\G$ are in bijection with $\Hh^1(\K, \W)$.
\end{enumerate}
\end{proposition}

Given a 1-cocycle $\rho \in \Zcycle^1(\K, \W)$ let $\Ss_\rho$ be the twist of $\Ss$ corresponding to $\rho$.  Since $E/\K$ is ramified, no $\Ss_\rho$ is actually unramified.  In order to work with those that are closest to unramified, we make the following definition.

\begin{definition}
We say that a torus $\Ss_\rho$ is \emph{relatively unramified} if it becomes isomorphic to $\Ss$ over the maximal unramified extension $\Ku$ of $\K$.
\end{definition}

\begin{proposition} \label{prop:ua_tori:unram_classifying_set} $ $
\begin{enumerate}
\item The torus $\Ss_\rho$ is relatively unramified if and only if the image of $\rho$ in $\Hh^1(\I, \W)$ is trivial.
\item Stable classes of relatively unramified maximal tori in $\G$ are in bijection with $\Hh^1(\GalKuK, \W^\I)$.
\end{enumerate}
\end{proposition}
\begin{proof}
Both statements follow from the inflation-restriction sequence \cite{serre:LocalFields}*{\S VII.6}
\[
0 \rightarrow \Hh^1(\GalKuK, \W^\I) \xrightarrow{\Inf} \Hh^1(\K, \W) \xrightarrow{\Res} \Hh^1(\I, \W)^{\Fr} \rightarrow \cdots
\]
and the fact that $\Hh^1(\I, \W)$ classifies stable classes of $\Ku$-tori.
\end{proof}

\subsection*{Anisotropic tori}

Suppose $\sigma \in \W^\I$ and let $\rho \in \Zcycle^1(\GalKuK, \W^\I)$ be the $1$-cocyle mapping Frobenius to $\sigma$.  We give a criterion for the torus $\Ss_\rho$ to be anisotropic and give a concrete description of $\Ss_\rho$ in this case.

\begin{proposition}
The torus $\Ss_\rho$ is anisotropic if and only if the signed cycle type of $\sigma$ has no positive cycles.
\end{proposition}
\begin{proof}
Since $E/\K$ is ramified, $\W^\I = \W^\GalK$ and Proposition \ref{prop:ua_tori:conj_in_Bn} applies.  Suppose the order of $\sigma$ is $r$, and write $\tilde{\rho}$ for the homomorphism $\GalK \to \W \rtimes \GalEK$ defined from $\rho$ by inflation.  The image of $\tilde{\rho}$ will be isomorphic to $\ZZ/r\ZZ \times \ZZ/2\ZZ$, generated by $(\sigma, 1)$ and $(1,\tau)$.  The characters in $X^*(\Ss_\rho)$ fixed by $(1,\tau)$ are spanned by $\chi_i - \chi_{-i}$ for $1 \le i \le m$.

If $\sigma$ has a positive cycle $C$ then $\sum_{i \in C} (\chi_i - \chi_{-i})$ provides a nonzero element of $X^*(\Ss_\rho)^{\GalK}$, so $\Ss_\rho$ cannot be anisotropic.

Conversely, if $\sigma$ is a product of disjoint negative cycles then we may decompose $X^*(\Ss_\rho)$ as a corresponding direct sum.  If $(C, -C)$ is a negative cycle then the only character fixed by the cycle is $\sum_{i \in C} (\chi_i + \chi_{-i})$, which is not fixed by $(1, \tau)$.  Moreover, for odd $n$, $\chi_0$ is negated by $(1, \tau)$.  Thus $X^*(\Ss_\rho)^{\GalK} = 0$ and $\Ss_\rho$ is anisotropic.
\end{proof}

\subsection*{Elemental tori} \label{sec:ua_tori:elem_tori}

The the decomposition of $\sigma \in \W^\I$ into negative cycles gives a corresponding decomposition of the torus $\Ss_\rho$ as a product of simpler tori.  We first give an intrinsic definition of these ``elemental'' tori, then prove the product decomposition, and finally study elemental tori in more detail.

For any $r$, let $\K_r$ be the unramified extension of $\K$ of degree $r$, and note that $\E_r = \E \cdot \K_r$ is an unramified extension of $E$ of degree $r$.
We have $\Gal(\E_r/\K) \cong \ZZ/r\ZZ \times \ZZ/2\ZZ$, since the element $\tau_r$ of order $2$ fixing $\K_r$ is central.  Let $\sigma_r$ be the image of $\Fr \in \GalE$ in $\Gal(\E_r/\K)$; it will be an element of order $r$ in $\Gal(\E_s/\K)$ fixing $\E$.

We will assume from now on that $s = 2r$ is even.  In this case, define $\eta_s = \tau_s \sigma_s^r$ and let $L_r$ be the fixed field of $\eta_s$.  The diagram of fields is:

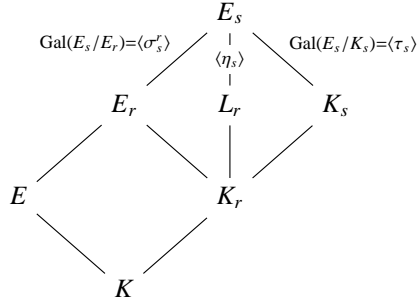
\[
\begin{tikzcd}[font=\normalsize,row sep={{{{35,between origins}}}},column sep={{{{40,between origins}}}},/tikz/commutative diagrams/arrows=dash]
 {} & & \E_s \arrow{dl}[swap]{\Gal(\E_s/\E_r)=\langle\sigma_s^r\rangle} \arrow{dr}{\Gal(\E_s/\K_s)=\langle\tau_s\rangle} \arrow{d}[description]{\langle \eta_s \rangle} & &\\
 {} & \E_r \dlar \drar & L_r \dar & \K_s \dlar \\
 \E \drar & & \K_r \dlar \\
 {} & \K
\end{tikzcd}
\]

We will frequently suppress the subscript and write $\tau$ for $\tau_s$, $\sigma$ for $\sigma_s$, $\eta$ for $\eta_s$ and $L$ for $L_r$.  Note that both $\tau \in \Gal(E_s/\K_s)$ and $\eta$ induce $\tau \in \Gal(E/\K)$ on $E$.

We define a torus $\T_s$ over $\K$ by 
\begin{equation} \label{eq:ua_tori:Tsdef}
\T_s = \Res_{L/\K} \U_1(E_s/L).
\end{equation}
These tori will form the building blocks for all maximal, relatively unramified, anisotropic tori:

\begin{theorem} \label{thm:ua_tori:elemental_decomposition}
Suppose that $\T \subset \G$ is a maximal, relatively unramified anisotropic torus, whose stable class corresponds to $\sigma \in \W^\I$.  Suppose $\sigma = \sigma_1\cdots \sigma_j$ is the decomposition of $\sigma$ into disjoint negative cycles, and let $s_i$ be the length of $\sigma_i$.
\begin{enumerate}
\item If $\G = \U_{2m+1}$, then $\T \simeq \prod_{i=1}^j \T_{s_i} \times \U_1$.
\item If $\G = \U_{2m}$, then $\T \simeq \prod_{i=1}^j \T_{s_i}$.
\end{enumerate}
\end{theorem}
\begin{proof}
In order to define each of these isomorphisms of tori, we may give a $\GalK$-equivariant isomorphism between $X^*(\T)$ and the character group of each right hand side.  The character group $X^*(\T_s)$ is easy to describe:
\[
X^*(\T_s) \cong \Ind^{\GalK}_{\GalL} X^*(\U_1(E_s/L)).
\]
Since $\U_1(E_s/L)$ splits over $E_s$, $\GalK$ acts on $X^*(\T_s)$ through its quotient $\Gal(E_s/\K)$.  Let $\{\chi\}$ be a basis for $X^*(\U_1(E_s/L)) \cong \ZZ$.  Each coset of $\Gal(E_s/L)$ in $\Gal(E_s/\K)$ contains a unique power of $\sigma$, and we can choose a basis for the induction where each basis function evaluates to $\chi$ on one power of $\sigma$ and zero on the others.  It will be convenient to denote this basis by $\{\chi_{-n}, \ldots, \chi_{-1}, \chi_1, \ldots, \chi_n\}$, where $\sigma$ acts by the cyclic permutation $\upsilon$ of $\{-n, \ldots, -1, 1, \ldots, n\}$ defined by

\[
\begin{tikzcd}[font=\normalsize,column sep={{{{35,between origins}}}}]
(-n) \arrow[out=315, in=225, looseness=0.4]{rrrr} & \cdots \lar & (-2) \lar & (-1) \lar & (1) \rar & (2) \rar & \cdots \rar & (n) \arrow[out=135,in=45,looseness=0.4]{llll}
\end{tikzcd}
\]

Then the action of $\Gal(E_s/\K)$ on $X^*(\T_s)$ is given by 
\begin{equation} \label{eq:ua_tori:Tsact}
\begin{aligned} 
\sigma(\chi_i) &= \chi_{\upsilon(i)} \\
\eta(\chi_i) &= -\chi_{i},
\end{aligned}
\end{equation}
and thus $\tau(\chi_i) = -\chi_{-i}$.

If $\sigma$ breaks up as the product of disjoint negative cycles of lengths $s_1, s_2, \ldots, s_k$ and $\G = \U_{2m}$, then $X^*(\T)$ will decompose as a representation of $\GalK$ into a direct sum of submodules of dimensions $2s_1, 2s_2, \ldots, 2s_k$, each spanned by the $\chi_i$ for $i$ occurring in a single negative cycle.  The action of $\GalK$ is precisely the one on $\T_s$ given in \eqref{eq:ua_tori:Tsact}.

The case that $\G=\U_{2m+1}$ is similar, but there will be an additional $1$-dimensional summand on which $\GalK$ acts through $\GalEK$, with $\tau$ negating $\chi_{0}$.
\end{proof}

\subsection*{Filtrations and N\'eron models}

Moy and Prasad define a decreasing filtration $\T_s(K)^r$ on $\T_s(K)$ (c.f. \cite{moy-prasad:96a} and \cite{yu:03a}*{\S 4-5}), where $\T_s(K)^0$ is given by the $\OK$-points of the identity component $\TT_s^\circ(\OK) \subset \TT_s(\OK) = \T_s(K)$ of the N\'eron model of $\T_s$.

\begin{proposition} \label{prop:ua_tori:Ts_components}
The N\'eron model $\TT_s$ of $\T_s$ is connected.
\end{proposition}
\begin{proof}
Formation of the N\'eron model commutes with Weil restriction \cite{bosch-lutkebohmert-reynaud:NeronModels}*{\S 7.6, Prop. 6}, and N\'eron models of unramified anisotropic tori such as $\U_1(E_s/L)$ are connected \cite{bosch-lutkebohmert-reynaud:NeronModels}*{\S 10.1, Prop. 3}.
\end{proof}

\begin{corollary}
The N\'eron model $\TT$ of a maximal, relatively unramified, anisotropic torus in $\U_{n}(E/\K)$ will be connected if $n$ is even and will have component group $\ZZ/2\ZZ$ if $n$ is odd.
\end{corollary}
\begin{proof}
This result follows from Theorem \ref{thm:ua_tori:elemental_decomposition} and Proposition \ref{prop:ua_tori:Ts_components}, together with the fact that the N\'eron model of $\U_1(E/\K)$ has two components.
\end{proof}

\section{Embedding tori in other unitary groups} \label{sec:embedding_tori}

For each elementary torus $\T_s$, we will define a family of Hermitian spaces $\{V_{s,\kappa}\}_{\kappa \in L^\times}$, together with an embedding of $\T_s$ into each unitary group $\U(V_{s,\kappa})$.  These unitary groups are not necessarily quasi-split.  Instead, we get embeddings into both pure inner forms of $\G$, which will eventually yield representations of the different pure inner forms.

As an $\E$-vector space, $V_{s,\kappa}$ is simply $\E_s$.  Following Euler (see \cite{serre:LocalFields}*{p. 56}), for any $\kappa \in \E_s$, define a bilinear form $\phi_\kappa$ on $V_{s,\kappa}$ by
\[
\phi_\kappa(x, y) = \Tr_{E_s/E}\left(\frac{\kappa}{\pi_L} \cdot x \cdot \eta(y)\right).
\]
Here $\pi_L$ is a uniformizer of $L$ with $\Tr_{L/K_r} \pi_L = 0$, and we divide by $\pi_L$ in the definition of $\phi_\kappa$ so that Proposition \ref{prop:induction:quasi-splitness} holds.

\begin{proposition}
The bilinear form $\phi_\kappa$ is Hermitian if and only if $\kappa \in L$.
\end{proposition}
\begin{proof}
The trace pairing is bilinear and nondegenerate.  Moreover, 
$\phi_\kappa(x, y) = \tau \phi_\kappa(y, x)$ if and only if $\kappa \in L$ since $\eta$ induces $\tau$ on $\E$.
\end{proof}

From now on we will assume that $\kappa \in L^\times$, in which case $V_{s, \kappa}$ is a Hermitian space.  Since 
\[
\T_s(\K) = \{\alpha \in \E_s \st \Nm_{E_s/L} \alpha = 1\},
\]
we have an embedding
\begin{align*}
\T_s(\K) &\rightarrow \U(V_{s, \kappa}), \\
\alpha &\mapsto \mbox{ multiplication by $\alpha$.}
\end{align*}

\begin{proposition} \label{prop:induction:quasi-splitness}
The unitary group $\U(V_{s, \kappa})$ is quasi-split if and only if $\kappa \in \Nm_{\E_s/L}(\E_s^\times)$.
\end{proposition}
\begin{proof}
We first reduce to the case $s = 2$.  Let $V_\kappa'$ be the two-dimensional $E_r$-vector space $E_s$ with Hermitian pairing $\phi_\kappa'$ (relative to the quadratic extension $\E_r/\K_r$) defined by
\[
\phi_\kappa'(x, y) = \Tr_{\E_s/\E_r}\left(\frac{\kappa}{\pi_L} \cdot x \cdot \eta(y)\right).
\]
We can reconstruct $\phi_\kappa$ from $\phi_\kappa'$ via the identity $\phi_\kappa = \Tr_{\E_r/\E} \circ \phi_\kappa'$.  We relate $V_{s,\kappa}$ to $V_\kappa'$ in the following lemma.

\begin{lemma}
The unitary group $\U(V_{s, \kappa})$ is quasi-split if and only if $\U(V_\kappa')$ is quasi-split.
\end{lemma}
\begin{proof}
If $\U(V_\kappa')$ is quasi-split then there is a nonzero isotropic vector $v \in V_\kappa'$.  One may check that $\E_r \cdot v$ is an $r$-dimensional isotropic subspace of $V_{s, \kappa}$, and thus $\U(V_{s,\kappa})$ is quasi-split.

Conversely, suppose that $X \subset V_{s, \kappa}$ is an $r$-dimensional isotropic subspace.  Since $\Tr_{\E_r/\E}$ is $\E$-linear, the set
\[
Y = \{y \in \E_r \st \Tr_{\E_r/\E}(y) = 0\}
\]
is an $(r-1)$-dimensional $E$-subspace of $E_r$.  The composition 
\[
X \xhookrightarrow{\phantom{abc}} \E_s \xlongrightarrow{\Delta} \E_s \times \E_s \xrightarrow{\phi_\kappa'} \E_r
\]
is $\K$-linear, and has image contained in $Y$.  But $\dim_\K X = 2r$ while $\dim_K Y = 2r-2$, so the composition has nontrivial kernel.  This yields a nonzero isotropic vector in $V_\kappa'$ and finishes the proof of the lemma.
\end{proof}

Suppose that $\kappa_1, \kappa_2 \in E_s^\times$ satisfy $\kappa_1 = \Nm_{\E_s/L}(\alpha) \kappa_2$.  Then $x \in \E_s$ is isotropic for $\phi_{\kappa_1}'$ if and only if $\alpha x$ is isotropic for $\phi_{\kappa_2}'$,
and $\U(V_{\kappa_1}')$ is quasi-split if and only if $\U(V_{\kappa_2}')$ is quasi-split.  So we need only consider $\kappa=1$ and $\kappa=\pi_L$, representatives for the two cosets of $\Nm_{E_s/L} \E_s^\times$ in $L^\times$.

If $\kappa = 1$ then $\pi_L \in V_1'$ is isotropic and thus $\U(V_{s,1})$ is quasi-split by the lemma.  An easy computation shows that when $\kappa = \pi_L$, the space $V_{\kappa}'$ has a different discriminant, so $\U(V_{s,\pi_L})$ is not quasi-split.
\end{proof}

Let $u \in \K^\times$ be a non-square unit (and thus $u \notin \Nm_{\E/\K} \E^\times$). 
\begin{corollary} \label{cor:induction:discVskappa}
We have $\disc(V_{s, \kappa}) \equiv u^{v_L(\kappa) + r (q-1)/2} \pmod{\Nm_{\E/\K} \E^\times}$.
\end{corollary}
\begin{proof}
Since $\E_s/L$ is unramified, $\kappa \in \Nm_{\E_s/L} \E_s^\times $ if and only if $v_L(\kappa) \equiv 0 \pmod{2}$.  By Proposition \ref{prop:background:Uclass}, the discriminant of the quasi-split unitary group of dimension $s$ is congruent to $(-1)^r$ modulo $\Nm_{\E/\K} \E^\times$.  Since $-1$ is a unit, it is a norm from $\E$ if and only if it is a square, which occurs if and only if $q \equiv 1 \pmod{4}$.  
\end{proof}

\subsection*{Embeddings of products}

We first consider even dimensional unitary groups.  By Theorem \ref{thm:ua_tori:elemental_decomposition} we may write $\T \simeq \prod_{i=1}^j \T_{s_i}$.  Let $\us = (s_1, \ldots, s_j) = (2r_1, \ldots, 2r_j)$ be the tuple of dimensions and set $L_i = L_{r_i}$.  

For odd dimensional unitary groups, Theorem \ref{thm:ua_tori:elemental_decomposition} implies that $\T \simeq \prod_{i=1}^{j-1} \T_{s_i} \times \U_1$.  Let $s_j = 1$ and $\us = (s_1, \ldots, s_j) = (2r_1, \ldots, 2r_{j-1}, 1)$ be the tuple of dimensions.  Set $L_i = L_{r_i}$ as above, and $L_j = \K$.  For $\kappa_j \in L_j^\times$ we can define a one-dimensional Hermitian space $V_{1, \kappa_j} \cong E$ by setting $\phi_{\kappa_j}(1,1) = \kappa_j / \pi_\K$.  We will write $\T_1$ for $\U_1$ to simplify notation: $\T_1(\K)$ acts on $V_{1, \kappa_j}$ by multiplication just as the other $T_{s_i}$ act on $V_{s_i, \kappa_i}$.  Note that the Hermitian condition on $\phi_{\kappa_j}$ forces $\kappa_j \in \K^\times$, and thus $v_E(\kappa_j)$ must be even.

In both cases we set $n = \sum_i s_i$.  For every $j$-tuple $\uk = (\kappa_1, \ldots, \kappa_j)$ with $\kappa_i \in L_i^\times$, we get a Hermitian space $V_{\us, \uk} = \prod_{i=1}^j V_{s_i, \kappa_i}$ and a product embedding $\T \hookrightarrow \U(V_{\us, \uk})$.  Write $\phi_{\uk}$ for the Hermitian pairing on $V_{\us, \uk}$, $\G_{\us, \uk}$ for $\U(V_{\us, \uk})$ and $\T_{\us, \uk}$ for the image of $\T$ in $\G_{\us, \uk}$.

\begin{proposition} \label{prop:induction:ukform}
For $n$ odd, $\G_{\us, \uk}$ is always quasi-split.  For $n$ even, $\G_{\us, \uk}$ is quasi-split if and only if 
\[
\sum_{i=1}^j v_L(\kappa_i) \equiv 0 \pmod{2}.
\]
\end{proposition}
\begin{proof}
The odd case is immediate since all odd unitary groups are quasi-split, so we assume that $n$ is even.  The discriminant of $\Hyp$ is $-1$, which is a norm from $E$ if and only if $q \equiv 1 \pmod{4}$.  Therefore the discriminant of a quasi-split space of dimension $2m$ will be $u^{(q-1)m}$.  The result now follows from Corollary \ref{cor:induction:discVskappa} and Proposition \ref{prop:disc_class}.
\end{proof}

\section{Bruhat-Tits buildings of unitary groups} \label{sec:buildings}

Suppose $\G$ is a reductive group over $\K$ with anisotropic center.  The Bruhat-Tits building $\BB(\G / \K)$ provides a tool for classifying models of $\G$ over $\OK$ and compact subgroups of $\G(\K)$.  Various structures on $\BB(\G)$ play a role in this classification:
\begin{itemize}
\item $\BB(\G)$ is a complete metric space and a simplicial complex.
\item $\G(\K)$ acts on $\BB(\G)$ by simplicial isometries.
\item $\BB(\G)$ is the union of a collection of distinguished subsets, known as \emph{apartments}, indexed by the maximal $\K$-split tori in $\G$.  The apartment $\Ap(\Aa)$ associated to $\Aa$ is an affine space for the real vector space $X_*(\Aa) \otimes \RR$.  When $\G$ is quasi-split, $\Aa$ is determined by its centralizer $\Ss = \Z_\G(\Aa)$ and we will also write $\Ap(\Ss)$ for $\Ap(\Aa)$.
\end{itemize}
A \emph{facet} is either a vertex or the interior of a positive-dimensional simplex, and an \emph{alcove} is a facet of maximal dimension.  For more details on Bruhat-Tits buildings, see Tits' introduction \cite{tits:79a}, the original articles by Bruhat and Tits \cites{bruhat-tits:72a, bruhat-tits:84a, bruhat-tits:84b, bruhat-tits:87a, bruhat-tits:87b}, Yu's survey article \cite{yu:09b}, or Garret's book \cite{garrett:BuildingsAndClassicalGroups} for buildings of split classical groups.

To each facet $F$ we may attach three subgroups of $\G(\K)$.  We write $\fix{F}$ for the subgroup that fixes every point of $F$ and $\stab{F}$ for the subgroup that stabilizes $F$.  These groups may be interpreted as the $\OK$-points of models $\GG_F$ and $\GG_F^\flat$ of $\G$.  These models have the same identity component $\GG_F^\circ$.  The parahoric subgroup $\parah{F} = \GG_F^\circ(\OK)$ is the third subgroup of $\G(\K)$ associated to $F$.  The Moy-Prasad filtration and the filtrations on the root spaces $\Uu_\alpha(\K)$ yield a filtration $\{\G(\K)_x^r\}_{r \ge 0}$ on the parahoric $\parah{F}$ \citelist{\cite{bruhat-tits:72a} \cite{bruhat-tits:84a} \cite{yu:03a}*{Thm. 8.3}}.  The depth-zero representations appearing in this paper arise via inflation along $\parah{x} \to \parah{x} / \G(\K)_x^{0+}$, so we need to understand these subgroups and their quotients.

We can give a concrete description of buildings of unitary groups in terms of lattices using the following theorem.
\begin{theorem}[{c.f. \cite{prasad-yu:02a}*{Thm. 1.9}}]
Suppose $\HHh$ is a connected, reductive group over a non-archimedian local field $\K$ and $\Omega$ is a finite group of $\K$-automorphisms of $\HHh$ whose order is not divisible by $p$.  Then $\G = (\HHh^\Omega)^\circ$ is reductive and $\BB(\G)$ can be identified with $\BB(\HHh)^\Omega$.
\end{theorem}
\noindent In particular, if $\HHh = \Res_{E/\K} \HHh'$ for some $\HHh'$ defined over $E$, then $\Omega = \GalEK$ acts on $\BB(\HHh / \K) = \BB(\HHh' / E)$.  Note that $p$ does not divide $\lvert \Omega \rvert$ precisely when $E / \K$ is tamely ramified.  Applying this theorem to the case that $\HHh'$ is $\U(V) / E \cong\GL_n / E$ we can realize $\BB(\U(V) / K)$ as the fixed points of $\BB(\HHh)$ under the involution induced by $\GalEK$.

\subsection*{\texorpdfstring{The building of $\GL_n$}{The building of GL(n)}}

For a vector space $V$ over $E$, we seek a concrete description of the building of $\GL(V)$ in order to describe $\BB(\U(V))$ as a subset.  The points of $\BB(\GL(V)/E)$ will be equivalence classes of norms on $V$, where we consider two norms equivalent if they differ by a constant \cite{yu:09b}*{\S 2.1.3}.

Norms are closely related to lattices.  A basis $\BBB = \{v_1, \ldots, v_n\}$ of $V$ is a \emph{splitting basis} for $\alpha$ if there exist $c_1, \ldots, c_n \in \RR$ so that
$\alpha\left(\sum \lambda_i v_i\right) = \min \left(v_E(\lambda_i) + c_i\right)$.  An apartment of $\BB(\GL(V))$ consists of those norm classes with a common splitting basis $\BBB$; the corresponding torus consists of the elements of $\GL(V)$ stabilizing the lines spanned by the vectors in $\BBB$.  The \emph{hyperspecial norm} associated to a basis $\BBB$ is the norm $\alpha\left(\sum \lambda_i v_i\right) = \min v_E(\lambda_i).$  To any hyperspecial norm $\alpha$ we associate the lattice $L_\alpha = \OE \langle v_1, \ldots, v_n \rangle$.
The equivalence relation on norms translates to one on lattices: $L$ and $L'$ are equivalent if $L' = \pi_E^c L$ for some $c \in \ZZ$.  The hyperspecial points in $\BB(\GL(V))$ are the vertices in the simplicial decomposition of $\BB(\GL(V))$ and correspond to equivalence classes of lattices.

A set of $k+1$ vertices form a simplex if there are lattices $L_0, \ldots, L_k$ representing the corresponding lattice classes such that
\[
L_0 \supsetneq L_1 \supsetneq \cdots \supsetneq L_k \supsetneq \pi_\K L_0.
\]

\subsection*{\texorpdfstring{The building of $\U_n$}{The building of U(n)}}

We now return to the analysis of the building of $\G = \U(V)$ in terms of an action of $\GalEK$ on the building of $\HHh = \Res_{E/\K} \GL(V)$.  Suppose that $\Aa$ is a maximal $\K$-split torus in $\U(V)$, contained in a maximal torus $\Ss$ that is defined over $\K$.  Since $\Ss$ is defined over $\K$, the apartment $\Ap(\Ss / E)$ is $\GalEK$-stable, and we can identify the $\GalEK$-fixed points with the apartment $\Ap(\Aa / \K)$.  If $\G$ is quasi-split, then each apartment of $\BB(\G)$ will be contained in a unique apartment of $\BB(\HHh)$; if $\G$ is not quasi-split then the dimension of the apartments of $\BB(\G)$ will be one less, and each apartment will be contained in many apartments of $\BB(\HHh)$.

Since $\GalEK$ acts on $\BB(\HHh)$ as a simplicial involution, there will be two types of simplices that intersect $\BB(\HHh)^{\GalEK}$:
\begin{itemize}
\item Simplices of $\BB(\HHh)$ that are fixed by $\GalEK$, corresponding to simplices of $\BB(\G)$ of the same dimension.  The Hermitian form $\phi$ on $V$ gives an identification of $V$ with its dual, and the dual of a lattice $\Lambda$ will be the lattice
\[
\Lambda^\vee = \{v \in V \st \phi(v, \lambda) \in \OE\}.
\]
The vertices of $\BB(\U(V))$ of this type correspond to lattices satisfying $\Lambda = \Lambda^\vee$.
\item Simplices of $\BB(\HHh)$ that are stabilized by $\GalEK$ but not fixed, each containing a simplex of $\BB(\G)$ of one lower dimension fixed by $\GalEK$.  Vertices of $\BB(\G)$ of this type arise from 1-simplices in $\BB(\HHh)$ whose ends are exchanged by the nontrivial element of $\GalEK$.  Such edges correspond to pairs of lattices $\Lambda_0$, $\Lambda_1$ with
\begin{align*}
\Lambda_0 &\supsetneq \Lambda_1 \supsetneq \pi_E \Lambda_0, \mbox{ or}\\
\Lambda_0^\vee &= \Lambda_1.
\end{align*}
\end{itemize}
Merging these two types, we see that vertices in the simplicial decomposition of $\BB(\U(V))$ correspond to lattice classes with a representative $
\Lambda$ satisfying
\[
\Lambda \supseteq \Lambda^\vee \supsetneq \pi_E \Lambda.
\]

\subsection*{Anisotropic tori}

In Section \ref{sec:embedding_tori}, we parameterized embeddings of the tori $\T_{\us, \uk}$ into unitary groups $\G_{\us, \uk}$.  Since $\T_{\us, \uk}(\K)$ is compact, it is contained in at least one maximal compact subgroup of $\G_{\us, \uk}$.  In fact, we may use its action on the building $\BB(\G_{\us, \uk})$ to see that $\T_{\us, \uk}(\K)$ is contained in a unique maximal compact subgroup.

\begin{theorem} \label{thm:induction:unique_T_fix_vertex}
The action of the torus $\T_{\us, \uk}(\K)$ fixes a unique vertex $x$ in $\BB(\G_{\us, \uk})$.
\end{theorem}
\begin{proof}
For the purpose of reducing subscripts, write $\E_i$ for $\E_{s_i}$, $L_i$ for $L_{r_i}$ and $\OO_i$ for the ring of integers of $\E_i$ for the duration of this proof.  For each tuple $\ub = (b_1, \ldots, b_j)$ of integers, we define a lattice 
\[
\Lambda_{\us, \ub} = \prod_{i=1}^j \pi_\E^{b_i} \OO_i \subset V_{\us, \uk}.
\]
Since $\Nm_{\E_i/L_i} \alpha = 1$ implies $\alpha \in \OO_i^\times$, the action of $\T_{\us, \uk}$ on $V_{\us, \uk}$ preserves $\Lambda_{\us, \ub}$.  Each extension $\E_i / \E$ is unramified and thus has trivial different, so the dual of $\OO_i$ under the trace pairing is just $\OO_i$, and the dual of $\pi_\E^{b_i}\OO_i$ under $\phi_{\kappa_i}$ is $\pi_\E^{-b_i-v_L(\kappa_i)} \OO_i$.  Therefore, if we write $v_L(\uk)$ for $(v_L(\kappa_1), \ldots, v_L(\kappa_j))$, 
\[
\Lambda_{\us, \ub}^\vee = \Lambda_{\us, -v_L(\uk) - \ub}.
\]
In order for 
\[
\Lambda_{\us, \ub} \supseteq \Lambda_{\us, \ub}^\vee \supseteq \pi_\E \Lambda_{\us, \ub},
\]
every entry of $-v_L(\uk) - 2\ub$ must be either $0$ or $1$.  There is a unique such $\ub$ for each $\uk$, and for this choice of $\ub$, the corresponding vertex of $\BB(\G_{\us, \uk})$ will be fixed by $\T_{\us, \uk}(\K)$.  In order to check that $\T_{\us, \uk}(\K)$ fixes a unique vertex, it suffices to check that any lattice fixed by $\T_{\us, \uk}(\K)$ must be one of the $\Lambda_{\us, \ub}$.

Suppose that $\Lambda$ is an $\OO_\E$-lattice in $V_{\us, \uk}$ fixed by $\T_{\us, \uk}(\K)$.  For each $i$ between $1$ and $j$, let $\lambda_i = (\lambda_{i,1}, \ldots, \lambda_{i,j}) \in \Lambda$ be any element with $v_{E_i}(\lambda_{i,i})$ minimal among the valuations of $i^\mathrm{th}$ coordinates of elements of $\Lambda$; let $b_i$ be this minimal valuation.  We now show that $\Lambda = \Lambda_{\us, \ub}$.

First we reduce to working one coordinate at a time.  Since $\Nm_{E_i/L_i}(-1) = 1$ for every $i$, we have an element 
\[
\alpha_i = (-1, -1, \ldots, -1, 1, -1, \ldots, -1) \in \T_{\us, \uk}(\K),
\]
where the $1$ occurs in position $i$.  Therefore we may replace $\lambda_i$ by 
\[
\lambda_i/2 + \alpha_i \lambda_i/2 = (0, \ldots, 0, \lambda_{i,i}, 0, \ldots, 0) \in \Lambda,
\]
which also has minimal valuation in the $i^\mathrm{th}$ coordinate.  

Write $\OE \cdot \T_s(\K)$ for the $\OE$-submodule of $\OO_{\E_s}$ generated by $\T_s(\K)$.  By our definition of the $b_i$, we have $\Lambda \subseteq \Lambda_{\us, \ub}$.  To show the reverse containment, it suffices to show that $\OE \cdot \T_s(\K) = \OO_{\E_s}$.  When $s=1$, this equation clearly holds.  For $s=2r$, we have the following lemma.

\begin{lemma} For any field $M$, let $\mu_n(M)$ denote the group of $n^\mathrm{th}$ roots of unity in $M$.  
\begin{enumerate}
\item $\mu_{q^r+1}(\E_s) \subset \T_s(\K)$,
\item $\mu_{q^r+1}(\E_s)$ generates $\OO_{\E_s}$ as an $\OO_\E$-module.
\end{enumerate}
\end{lemma}
\begin{proof}
The nontrivial element of $\Gal(\E_s/L)$ is $\Fr^r$, which acts on elements $\alpha \in \mu_{q^r+1}$ by $\alpha \mapsto \alpha^{q^r}$.  Thus $\Nm_{\E_s/L}(\alpha) = \alpha^{q^r+1} = 1$, so $\alpha \in \T_s(\K)$.

Now let $\k_s$ be the degree $s$ extension of $\k$ and $\bar{\alpha}$ be a generator for the cyclic group $\mu_{q^r+1}(\k_s)$.  Since the multiplicative order of $\bar{\alpha}$ is $q^r+1$, $\bar{\alpha}$ is not contained in any subfield of $k_s$, and thus the set $\{1, \bar{\alpha}, \ldots, \bar{\alpha}^{s-1}\}$ is a basis for $\k_s$ over $\k$.  Since $\E_s/\E$ is unramified we can approximate any element of $\OO_{\E_s}$ arbitrarily well with elements of $\OO_\E \cdot \T_s(\K)$.  Completeness of $\OO_{\E_s}$ now finishes the proof.
\end{proof}

Returning to the proof of the theorem, we have shown that $\Lambda = \Lambda_{\us, \ub}$ for an appropriate choice of $\ub$.  Therefore the action of $\T_{\us, \uk}(\K)$ fixes a unique vertex.
\end{proof}

\begin{corollary} \label{cor:induction:unique_fixed_point_on_building}
The torus $\T_{\us, \uk}(\K)$ fixes no other point in $\BB(\G_{\us, \uk})$.
\end{corollary}
\begin{proof}
Suppose that $\T_{\us, \uk}(\K)$ fixes an additional point $y \in \BB(\G_{\us, \uk})$, which we may take to lie in a common apartment $\Ap$.  Since $\T_{\us, \uk}(K)$ acts isometrically, it must fix the whole line between $x$ and $y$.  This line will pass through the interior of some facet in $\Ap$ that is not a vertex.  Since $\T_{\us, \uk}(\K)$ acts by simplicial automorphisms, it must fix the whole facet, and thus the vertices in the closure of the facet.  This contradicts Theorem \ref{thm:induction:unique_T_fix_vertex}.
\end{proof}

\begin{corollary} \mbox{}
\begin{enumerate}
\item $\T_{\us, \uk}(\K)$ is contained in a unique maximal compact subgroup $\GG_{\us, \uk}(\OK) \subset \G_{\us, \uk}(\K)$.
\item $\TT_{\us, \uk}^\circ(\OK)$ is contained in a unique maximal parahoric subgroup $\GG_{\us, \uk}^\circ(\OK)$.
\end{enumerate}
\end{corollary}
\begin{proof}
Every maximal compact subgroup fixes a point of $\BB(\G)$, and every maximal parahoric subgroup fixes a vertex.
\end{proof}

At this point we fix $\us$ and $\uk$ in order to simplify the notation.  Note that $\us$ is determined by $\T$, and the choice of $\uk$ is equivalent to a choice of embedding $\T \hookrightarrow \G'$ for some inner form $\G'$ of $\G$.  We set
\begin{align*}
G &= \G_{\us, \uk}(\K),  \\
G^\flat &= \GG_{\us, \uk}^\flat(\OK) = \GG_{\us, \uk}(\OK), \\
G^\circ &= \GG_{\us, \uk}^\circ(\OK),\\
\Gbar^* &= \GG_{\us, \uk}^\flat(\k) = \GG_{\us, \uk}(\k).\\
\end{align*}
Finally, let $\Gbar$ be the maximal reductive quotient of $\Gbar^*$, and let $\Gbar^\circ$ the connected component of the identity of $\Gbar$.

\subsection*{Reductions of parahorics and maximal compacts}

Our construction of representations of $G$ has as intermediate steps the construction of representations of $G^\circ$ and then $G^\flat$.  We need to understand the reductions of $G^\circ$ and $G^\flat$ in order to pass from a representation of the first to a representation of the second.  

We may assume that $\uk$ is sorted so that all of the $\kappa_i$ with odd valuation appear at the beginning and those of even valuation at the end.  If $n$ is odd this convention aligns with our previous choice of putting the $\U_1$ last, since $v_E(\kappa_j)$ will always be even.  Let $d$ be the cutoff so that $\kappa_d$ has odd valuation and $\kappa_{d+1}$ even valuation.  Let $l = \sum_{i=1}^d s_i$ and $m = \sum_{i=d+1}^j s_i$.

The filtration on the parahoric $G^\circ$ induces a filtration on $G^\flat$, and the quotient
\[
\Gbar = \GG_{\us, \uk}(\OK)/\GG_{\us, \uk}(\OK)^{0+}
\]
gives the $\k$-points of a reductive group over $\k$.

\begin{theorem} \label{thm:induction:Ured}
Suppose that $\G = \U_n / \K$ is a unitary group.
\begin{enumerate}
\item The reduction $\Gbar$ is given by 
\[
\Gbar \cong \Sp_l(\k) \times \Orth_m(\k).
\]
\item The connected component of the identity is given by 
\[
\Gbar^\circ \cong \Sp_l(\k) \times \SO_m(\k).
\]
\end{enumerate}
\end{theorem}
\begin{proof}
Let 
\[
\Lambda = \prod_{i=1}^j\pi_E^{b_i}\OO_i
\]
be the lattice corresponding to the vertex fixed by $\G^\flat$ as in the proof of Theorem \ref{thm:induction:unique_T_fix_vertex}.  By our definitions of $l$ and $m$, the first $d$ entries of $v_E(\uk)+2\ub$ are $-1$ and the last $j-d$ are $0$.  Set
\[
\bar{\Lambda} = \Lambda / \pi_\K \Lambda.
\]
Since $G$ stabilizes the lattice $\Lambda$, we get an action of $G$ on $\bar{\Lambda}$.
Note that $\GG_{\us, \uk}(\OK)^{0+}$ acts trivially on $\bar{\Lambda}$, and thus we get an action of $\Gbar$ on $\bar{\Lambda}$.

Following Tits \cite{tits:79a}*{\S 3.11}, we consider the endomorphism $\nu$ of $\bar{\Lambda}$ induced by multiplication by $\pi_E$ within $\Lambda$.  The endomorphism $\nu$ is clearly centralized by the action of $\Gbar^*$, and has kernel equal to its image.  Set
\[
\bar{\Lambda}_0 = \bar{\Lambda} / \nu(\bar{\Lambda}) \cong \Lambda / \pi_E \Lambda.
\]
Since $\Gbar^*$ centralizes $\nu$, we get a homomorphism $\Gbar^* \rightarrow \GL(\bar{\Lambda}_0)$ with unipotent kernel.  

The skew Hermitian form $\pi_E \phi_\kappa$ takes integral values on $\Lambda$ since $\Lambda^\vee \supseteq \pi_E \Lambda$, and thus induces an alternating form $\bar{\phi}_0$ on $\bar{\Lambda}_0$.  This form is degenerate, with kernel $\bar{\Lambda}_1 \subset \bar{\Lambda}_0$ equal to the image of $\Lambda^\vee$ in $\bar{\Lambda}_0$.  The dimension of $\bar{\Lambda}_1$ is the sum of the dimensions of the components of $\Lambda$ corresponding to $\kappa_i$ with even valuation, namely $\dim_\k(\bar{\Lambda}_1) = m$.  Our alternating form induces a nondegenerate alternating form on the quotient $\bar{\Lambda}_0 / \bar{\Lambda}_1$, a $\k$-vector space of dimension $l$.

The Hermitian form $\phi_\kappa$ takes integral values on $\Lambda^\vee$ since $\Lambda^\vee \subseteq \Lambda$, and thus induces a symmetric form $\phi_1$ on $\bar{\Lambda}_1$.  The image of $\Gbar^*$ in $\GL(\bar{\Lambda}_0)$ preserves these two forms, and the maximal reductive quotient $\Gbar$ is just the product
\[
\Gbar = \Sp(\phi_0) \times \Orth(\phi_1).
\]

The second half of the theorem now follows easily.
\end{proof}

\begin{corollary}
The size of the component group of $G^\flat$ is given by
\[
\eqnchoice{\lvert G^\flat / G^\circ \rvert}{1}{if $n$ is even and all $\kappa_i$ have odd valuation}{2}{otherwise}
\]
\end{corollary}

In the case that $G^\circ$ sits inside $G^\flat$ with index $2$, we will need to determine whether the induction of a Deligne-Lustig representation remains irreducible after inducing.  To this end, we have the following proposition.

\begin{proposition} \label{prop:induction:connected_center}
The center $\Z(\barG)$ lies within $\barG^\circ$ if and only if $n$ is even.
\end{proposition}
\begin{proof}
Since $\barG = \Sp_l(\k) \times \Orth_m(\k)$, since $n$ has the same parity as $m$, and since $\Sp_l(\k)$ is connected, it suffices to prove the statement for $\barG = \Orth_n(\k)$.

In order for a diagonal matrix $\alpha$ to be orthogonal, we must have $\alpha^2 = 1$.  For scalar $\alpha$, this condition reduces to $\alpha = \pm 1$.

If $n$ is odd, the $-1$ matrix does not lie in $\SO_n(\k)$ but does lie in the center of $\Orth_n(\k)$.  For $n$ even, $-1 \in \SO_n(\k)$ and thus $\Z(\barG) \subset \barG^\circ$.
\end{proof}

Note that the different reductions line up correctly with the reductions given in Appendix \ref{app:loc_indx}, Figure \ref{fig:background:Ureductions}.  In particular, if $n=2m$ and $G$ is quasi-split, then there must be either no odd $v_{E_i}(\kappa_i)$ or at least two; this explains why there are no reductions of the form $\Orth_2 \times \Sp_{2m-2}$ for the quasi-split $G$.  Conversely, if $G$ is not quasi-split then there must be at least one odd $v_{E_i}(\kappa_i)$, corresponding to the lack of any reduction of the form $\Sp_{2m}$.

In the other direction, Figure \ref{fig:background:Ureductions} gives us information about the orthogonal form $\phi_1$ in the proof of Theorem \ref{thm:induction:Ured}: it will be split if $G$ is quasi-split and non-split otherwise.

\section{Tori from Langlands parameters} \label{sec:ua_tori:construction}

Let $\G = U(V)$ be a quasi-split unitary group and $\varphi : \Weil_K \to {}^L\G$ be a tame, discrete, regular Langlands parameter as in the introduction.  We defined in \S \ref{sec:tori} a quasi-split torus $\Ss \subset \G$, unique up to conjugacy.  In this section we will construct from $\varphi$ a maximal, relatively unramified, anisotropic torus $\T$ that will serve as an ingredient for the representations in the $L$-packet $\Pi_\varphi$.

Choose a topological generator $\ttau$ of the tame inertia group $\Itame$ with image $\tau \in \GalEK$, and define $z \in \hat{\G}$ by
\[
\varphi(\ttau) = z\tau.
\]

\begin{proposition} \label{prop:phiconj}
We may conjugate $\varphi$ by an element of $\hat{\G}$ so that $z \in \hat{\Ss}^\tau$.
\end{proposition}
\begin{proof}
An automorphism of $\hat{\G}$ is said to be semisimple if its action on $\hat{\g}$ is diagonalizable.  Since $\varphi(\ttau)$ has finite order, conjugation by it is a semisimple automorphism.  We now apply \cite{reeder:10a}*{Lem. 3.2}.
\end{proof}

From now on, we assume that $z \in \hat{\Ss}^\tau$.  To construct our unramified anisotropic torus, we want to obtain an elliptic element of $\W^\I$.  The first step in this process is the following lemma:

\begin{lemma} \label{lem:ua_tori:cent_ttau}
Assume that $\varphi$ is regular.  Then the centralizer of $\varphi(\ttau)$ is given by 
\[
\Z_{\hat{\G}}(\varphi(\ttau)) = \hat{\Ss}^\tau.
\]
\end{lemma}
\begin{proof}
The group $\Z_{\hat{\G}}(\varphi(\ttau))$ certainly contains $\hat{\Ss}^\tau$.  By Proposition \ref{prop:phiconj}, conjugation by $\varphi(\ttau)$ stabilizes $\hat{\Ss}$ and thus $\hat{\Ss}^\tau$ is a maximal torus in $\Z_{\hat{\G}}(\varphi(\ttau))$.  But our assumption that $\varphi$ is regular implies that $\Z_{\hat{\G}}(\varphi(\ttau))$ is a torus, and we thus obtain the desired result.

Alternatively, one can use a result of Reeder \cite{reeder:10a}*{Prop. 3.8} to equate the Lie algebras $\hat{\g}^{\varphi(\ttau)}$ and $\hat{\s}^\tau$. 
\end{proof}

In order to get an relatively unramified torus, we need an element of $\Hh^1(\GalKuK, \W^\I)$.  We do so by reducing $\varphi(\Fr)$ modulo $\hat{\Ss}$.  As long as $\varphi$ maps $\Fr$ into $\N_{\hat{\G}}(\hat{\Ss})$, we obtain a cocycle in $\Hh^1(\GalKuK, \W^\I)$.

\begin{proposition} \label{prop:ua_tori:Frob_in_N}
After conjugating so that $z \in \hat{\Ss}^\tau$, we have $\varphi(\Fr) \in \N_{\hat{\G}}(\hat{\Ss})$.
\end{proposition}

\begin{proof}
We first show that it suffices to find a regular element $z_0 \in \hat{\Ss}^\tau$.  The centralizer of $z_0$ would then would be the unique maximal torus containing $z_0$ \cite{humphreys:95a}*{Prop. 2.3} and would also contain $\hat{\Ss}^\tau$:
\begin{equation} \label{eq:ua_tori:ZSt}
\Z_{\hat{\G}}(\hat{\Ss}^\tau) = \hat{\Ss}.
\end{equation}
The image of $\Fr$ under $\varphi$ must normalize $\hat{\Ss}^\tau$ since $\Fr$ normalizes the powers of $\ttau$.  Now \eqref{eq:ua_tori:ZSt} implies that $\varphi(\Fr) \in \N_{\hat{\G}}(\hat{\Ss})$.

To find $z_0$, let $2\rho^\vee$ be the sum of the positive coroots of $\hat{\Ss}$ in $\hat{\G}$, which is $\tau$-invariant since the corresponding Borel subgroup of $\hat{\G}$ is stable under $\tau$.  We claim that for $\epsilon \ne 0$, $z_0 = \rho^\vee(1+\epsilon)$ is an element of $\hat{\Ss}^\tau$ and a regular semisimple element of $\hat{\G}$.  The first claim follows since $\rho^\vee$ is $\tau$-invariant, and the second since no root of $\hat{\Ss}$ vanishes on $\rho^\vee$.
\end{proof}

Proposition \ref{prop:ua_tori:Frob_in_N} allows us to define an element $\omega \in \W \cong \N_{\hat{\G}}(\hat{\Ss}) / \hat{\Ss}$ by projecting $\varphi(\Fr)$.  Since $\varphi$ and thus $\G$ are tamely ramified, $q$ must be odd.  Therefore the projection $\tau \in \W \rtimes \GalEK$ of $\varphi(\ttau)$ will satisfy $\omega \tau \omega^{-1} = \tau^q = \tau$, and thus $\omega \in \W^\I$.  By Proposition \ref{prop:ua_tori:unram_classifying_set} we get an isomorphism class of relatively unramified tori.  We will denote by $\T$ an abstract torus in this isomorphism class.  Moreover, since we assume that the centralizer of the image of $\varphi$ is finite, Frobenius acts without fixed points on $X_*(\Ss^\tau)$.  Thus $\omega$ is an elliptic element of $\W^\I$ and $\T$ is anisotropic. Tracing through the bijection between $\Hh^1(\GalKuK, \W^\I)$ and stable classes of tori, we can describe the Galois action on $X_*(\T)$.

\begin{construction} \label{con:ua_tori:T_construction}
The construction described in this section produces a maximal, relatively unramified, anisotropic torus $\T$.  The splitting field $M$ of $\T$ is naturally identified with the subgroup of $\W^\I \times \GalEK$ generated by $\omega$ and $\GalEK$.  The character and cocharacter groups $X^*(\T)$ and $X_*(\T)$ are identified with $X^*(\Ss)$ and $X_*(\Ss)$. While the action of $\ttau$ remains the same, Frobenius now acts via $\omega$ rather than trivially.
\end{construction}

We can summarize the action of $\GalK$ on $\hat{\T}$ as follows.  As a complex algebraic group, we identify $\hat{\Ss}$ with $\hat{\T}$.  Let $D_\varphi$ be the subgroup of ${}^L\G$ generated by $\hat{\T} \rtimes \GalEK$ and $\varphi(\Fr)$.  Then there is an exact sequence
\begin{equation}
1 \to \hat{\T} \to D_\varphi \to \GalMK \rightarrow 1
\end{equation}
so that the action of $\GalMK$ on $\hat{\T}$ is given by conjugating by a lift in $D_\varphi$.

\section{\texorpdfstring{Groups of type $L$}{Groups of type L}} \label{sec:character:typeL}

The group $D_\varphi$ is an example of a group of type $L$, a notion generalizing $L$-groups.

\begin{definition}
Suppose $\T$ is a torus with splitting field $M/\K$.  A \emph{group of type $L$} associated to $\T$ is an extension $D$ of the form
\[
1 \rightarrow \hat{\T} \rightarrow D \rightarrow \Gal(M/\K) \rightarrow 1.
\]
Such extensions are classified up to isomorphism by $\Hh^2(\Gal(M/\K), \hat{\T})$.

A \emph{split group of type $L$} is a group of type $L$ together with a chosen section of $D \rightarrow \Gal(M/\K)$ that yields an isomorphism $D \cong \hat{\T} \rtimes \Gal(M/\K)$.
\end{definition}

The notion of a group of type $L$ is similar to that of Vogan's weak extended group for $\G$ \cite{vogan:93a}*{Def. 2.3}, but with a torus $\T$ in place of a more general reductive group $\G$, and with $\Gal(M/\K)$ in place of $\GalK$.

For any group $D$ of type $L$, let $P_\K(D, \T)$ denote the set of equivalence classes of homomorphisms from $\Weil_\K$ to $D$ that yield the standard projection $\Weil_\K \rightarrow \Gal(M/\K)$ when composed with $D \rightarrow \Gal(M/\K)$.  We consider two such homomorphisms equivalent if one can be obtained from the other via conjugation by an element of $\hat{\T}$.  One can consider $P_\K(D, \T)$ to be a generalization of $\Hh^1(\K, \hat{\T})$, since in the case that $D$ is split
\[
P_\K(D, \T) = \Hh^1(\K, \hat{\T}).
\]
Note that there is no natural group structure on $P_\K(D, \T)$ when $D$ is not split.

\subsection*{Restriction}

For any extension $N$ of $\K$, we can consider the extension of scalars $\T_N$ of $\T$ to $N$.  The splitting field of $\T_N$ is just given by $NM$, and we have $\Gal(NM/N) \cong \Gal(M/N \cap M)$.  If we have a group $D$ of type $L$ associated to $\T$, then we can obtain a group $D_N$ associated to $\T_N$ as follows.  The dual group $\widehat{\T_N}$ is just $\hat{\T}$ with the subgroup $\GalN \subset \GalK$ acting, and thus we will denote it as $\hat{\T}$ as well.  We can thus define $D_N \subseteq D$ as the inverse image of $\Gal(M/N \cap M) \subseteq \Gal(M/\K)$.  The canonical isomorphism between $\Gal(NM/N)$ and $\Gal(M/N \cap M)$ then gives us the exact sequence
\begin{equation} \label{groupL-res-seq}
1 \rightarrow \hat{\T} \rightarrow D_N \rightarrow \Gal(NM/N) \rightarrow 1.
\end{equation}

We can now define a restriction map 
\[
\res_{N/\K} \colon P_\K(D, \T) \rightarrow P_N(D_N, \T_N)
\]
by just restricting to $\GalN$.  If $D$ is split then $D_N$ will be split by the restriction of the splitting map $\Gal(M/\K) \rightarrow D$ to $\Gal(M/M \cap N)$.  In this case, $\res_{N/\K}$ is just the normal restriction map of group cohomology from $\Hh^1(\K, \hat{\T}) \rightarrow \Hh^1(N, \hat{\T})$.

\begin{theorem} \label{thm:character:res_ker}
Suppose that $\K_f/\K$ is the maximal unramified subextension of the splitting field $M/\K$ of $\T$.  Then each fiber of 
\[
\res_{\K_f/\K} \colon P_\K(D, \T) \rightarrow P_{\K_f}(D_{\K_f}, \T_{\K_f})
\]
is either empty or a principal homogeneous space for $\Hh^1(\K, \hat{\T}^\I)$.
\end{theorem}
\begin{proof}
For a fixed $\varphi \in P_\K(D, \T)$, let $\varphi_f = \res_{\K_f/\K}(\varphi)$.  We want to describe the set of all $\varphi'$ with $\res_{\K_f/\K}(\varphi') = \varphi_f$.  Since $\K_f/\K$ is unramified, in order to extend $\varphi_f$ to all of $\GalK$, we need only specify the image of some Frobenius element $F \in \GalK$.  By multiplying $F$ by an element of $\I$ if necessary, we may assume that $F^f$ acts trivially on $M$.  Since $F^f \in \GalM$, we must therefore have $\varphi_f(F^f) \in \hat{\T}$.  Whatever value $\varphi'(F)$ takes, it must satisfy
\begin{equation} \label{eq:character:fth_power}
\varphi'(F)^f = \varphi_f(F^f) = \varphi(F)^f.
\end{equation}
Write $x = \varphi(F)$, $x' = \varphi'(F)$ and $y = x^{-1}x'$.  Since $x$ and $x'$ have the same image in $\Gal(M/\K)$, we have $y \in \hat{\T}$.  Moreover, using the conjugation action of Frobenius on inertia, it is straightforward to show that $y$ commutes with $\varphi_f(\alpha)$ for any $\alpha \in \I$ and thus $y \in \hat{\T}^\I$.  If we define $\Nm_{\hat{\T}^\I} \colon \hat{\T}^\I \rightarrow \hat{\T}^\I$ by $t \mapsto \prod_{j=1}^f F^j(t)$, then one may further show that $y \in \ker(\Nm_{\hat{\T}^\I}$ using \eqref{eq:character:fth_power}. 

Conversely, suppose that $y \in \ker(\Nm_{\hat{\T}^\I})$.  Then setting $x' = xy$ and working backward through the same steps we find that $x'$ satisfies all the identities required for the image of $F$, and thus defines an element $\varphi' \in P_\K(D, \T)$ with the same restriction to $P_{\K_f}(D_{\K_f},\T_{\K_f})$.

Since elements of $P_\K(D, \T)$ are only defined up to $\hat{\T}$ conjugacy, different values of $x'$ may yield the same element.  In fact, $x$ and $x'$ will yield the same element of $P_\K(D,\T)$ if and only if $y \in (F-1) \hat{\T}^\I$.  To show this fact, suppose that $y = F(z) z^{-1}$ for some $z \in \hat{\T}^\I$.  Since the image of $\varphi_f$ projects onto $\Gal(M/\K_f)$, each element commutes with $F(z) \in \hat{\T}^\I$ and thus conjugating by $F(z)$ leaves the restriction $\varphi_f$ fixed.  Therefore, $xy$ yields the same element of $P_\K(D, \T)$ as 
\[
F(z)xyF(z)^{-1} = F(z) x z^{-1} = x.
\]

Conversely, suppose $x$ and $xy$ are identified after conjugating by some element $z \in \hat{\T}$.  We have already fixed $\varphi_f$, so $z$ must commute with every element of the image of $\varphi_f$.  Since the image projects surjectively onto $\Gal(M/\K_f)$ we must in fact have $z \in \hat{\T}^\I$.  Finally, $xy = x \cdot F(F^{-1}(z)^{-1}) \cdot F^{-1}(z),$ and thus $y \in (F - 1)\hat{\T}^\I$.

We finish the proof by noting that $\Hh^1(\K, \hat{\T}^\I) \cong \ker(\Nm_{\hat{\T}^\I}) / (F-1)\hat{\T}^\I$.
\end{proof}

In the case that $D$ is split, theorem \ref{thm:character:res_ker} reduces to the inflation-restriction sequence \cite{serre:LocalFields}*{Prop. VII.6.4}:
\[
1 \rightarrow \Hh^1(\Gal(\K_f/\K), \hat{\T}^\I) \xrightarrow{\inf} \Hh^1(\K, \hat{\T}) \xrightarrow{\res_{\K_f/\K}} \Hh^1(\K_f, \hat{\T}).
\]
By the remark at the end of that section, this sequence extends to 
\begin{multline} \label{eq:character:inf_res_full}
1 \longrightarrow \Hh^1(\Gal(\K_f/\K), \hat{\T}^\I) \xlongrightarrow{\inf} \Hh^1(\K, \hat{\T}) \xlongrightarrow{\res} \Hh^1(\K_f, \hat{\T})^{\Gal(\K_f/\K)} \\
 \longrightarrow \Hh^2(\Gal(\K_f/\K), \hat{\T}^\I) \longrightarrow \Hh^2(\K, \hat{\T}).
\end{multline}
We next seek an analogue for the second part of this sequence when $D$ is not split.

\subsection*{\texorpdfstring{Relatively unramified groups of type $L$}{Relatively unramified groups of type L}}

Since $\GalK$ acts only on $\hat{\T}$ and not on all of $D_{\K_f}$, we cannot merely follow Serre \cite{serre:LocalFields}*{Prop. VII.6.3} to define an action of $\Gal(\K_f/\K)$ on $P_{\K_f}(D_{\K_f}, \T_{\K_f})$ in general.  However, if the sequence \eqref{groupL-res-seq} is split then $\Gal(\K_f/\K)$ does act on $P_{\K_f}(D_{\K_f}, \T_{\K_f}) \cong \Hh^1(\K_f, \hat{\T})$.  We may extend this action to a broader class of groups of type $L$.

\begin{definition}
We say that $D$ is \emph{relatively unramified} if $D_N$ is split for some unramified extension $N/\K$.
\end{definition}

\begin{proposition} \label{prop:character:unramified_equiv}
Set $I = \Gal(M/\K_f)$.  The following conditions are equivalent:
\begin{enumerate}
\item $D$ is relatively unramified,
\item there is a function $\iota \colon I \rightarrow D$ splitting the map $D \rightarrow \Gal(M/\K)$,
\item there is an exact sequence
\[
1 \rightarrow \hat{\T} \rtimes I \rightarrow D \rightarrow \Gal(\K_f/\K) \rightarrow 1
\]
compatible with the one defining $D$.
\end{enumerate}
\end{proposition}
\begin{proof} \mbox{}
\begin{itemize}
\item $(i) \Rightarrow (ii)$: If $D$ is relatively unramified, then there is a subfield $N \subset \K_f$ and a homomorphism $\iota' \colon \Gal(M/N) \rightarrow D_N$ splitting the sequence
\[
1 \rightarrow \hat{\T} \rightarrow D_N \rightarrow \Gal(M/N) \rightarrow 1.
\]
The restriction of $\iota'$ to $I$ yields the desired splitting of $D \rightarrow \Gal(M/\K)$.
\item $(ii) \Rightarrow (iii)$: If we identify $I$ with its image under $\iota$, we get a subgroup $\hat{\T} \rtimes I \subset D$.  The quotient is just $\Gal(M/K) / I \cong \Gal(\K_f/\K)$.
\item $(iii) \Rightarrow (i)$: The restriction of $\hat{\T} \rtimes I \rightarrow D$ to $I$ provides a splitting for $N = \K_f$.
\end{itemize}
\end{proof}

We can now return to the exact sequence \eqref{eq:character:inf_res_full}.

\begin{proposition} \label{prop:character:res_im_fixed}
Suppose that $D$ is relatively unramified, and that $\Gal(M/\K_f)$ is abelian.
Then the image of 
\[
 \res_{\K_f/\K} \colon P_\K(D, \T) \rightarrow P_{\K_f}(D_{\K_f}, \T_{\K_f})
\]
 is fixed by $\Gal(\K_f/\K)$.
\end{proposition}
\begin{proof}
Since $D$ is relatively unramified, $P_{\K_f}(D_{\K_f}, \T_{\K_f}) \cong \Hh^1(\K_f, \hat{\T})$, and we want to use the standard action of $\Gal(\K_f/\K)$ on $\Hh^1(\K_f, \hat{\T})$ to give an action of $\GalK$ on $P_{\K_f}(D_{\K_f}, \T_{\K_f})$.  We would like to define, for $\varphi \in P_{\K_f}(D_{\K_f}, \T_{\K_f})$, $\sigma \in \GalK$ and $\epsilon \in \GalKf$
\[
(\sigma . \varphi)(\epsilon) = \sigma . \varphi(\sigma^{-1} \epsilon \sigma).
\]
Here the action of $\GalK$ on $\hat{\T} \rtimes I$ should come from conjugation within $D$, using the exact sequence from Proposition \ref{prop:character:unramified_equiv}: to determine how $\sigma$ acts we first project it to $\Gal(\K_f/\K)$, then lift it arbitrarily to $D$ and conjugate.  This does not actually yield an action on $\hat{\T} \rtimes I$, since the action would depend on our choice of lift.  Suppose that $x$ and $x'$ are two different lifts, and thus $x' = (t, i)x$ for some $(t, i) \in \hat{\T} \rtimes I$.  Then conjugation by $x$ and by $x'$ differs by conjugation by $(t, i)$.  Since $I$ is assumed to be abelian, a simple computation shows that conjugating by $(t, i)$ is the same as conjugating by $t \in \hat{\T}$.  Thus the ambiguity in the definition of the action of $\Gal(\K_f/\K)$ on $\hat{\T} \rtimes I$ disappears once we note that elements of $P_{\K_f}(D_{\K_f}, \T_{\K_f})$ are defined up to conjugation by an element of $\hat{\T}$.  Similarly, modifying $\sigma$ by an element of $\GalKf$ has the effect of conjugating $\varphi(\sigma^{-1} \epsilon \sigma)$ by an element of $\hat{\T} \rtimes I$, and thus by an element of $\hat{\T}$ by the same reasoning.  So we get a genuine action of $\Gal(\K_f/\K)$ on $P_{\K_f}(D_{\K_f}, \T_{\K_f})$, and one can check that in fact this is the same action as the one on $\Hh^1(\K_f, \hat{\T})$ described in \cite{serre:LocalFields}*{Prop. VII.6.3}.

If $\varphi \in P_{\K_f}(D_{\K_f}, \T_{\K_f})$ is in the image of restriction, write $\tilde{\varphi}$ for a homomorphism on $\GalK$ with restriction $\varphi$.  Then for $\sigma \in \GalK$,
\begin{align*}
(\sigma . \varphi)(\epsilon) &= \sigma . \varphi(\sigma^{-1} \epsilon \sigma) \\
&= \tilde{\varphi}(\sigma) \tilde{\varphi}(\sigma^{-1}) \varphi(\epsilon) \tilde{\varphi}(\sigma) \tilde{\varphi}(\sigma)^{-1}\\
&= \varphi(\epsilon),
\end{align*}
where all equalities are defined up to conjugation by an element of $\hat{\T}$ that depends on $\sigma$ but not $\epsilon$.
\end{proof}

Every group of type $L$ associated to a tame Langlands parameter will be relatively unramified:

\begin{proposition} \label{prop:character:Dphi_unram}
Suppose that $\varphi$ is a tame, discrete, regular Langlands parameter.  Then $D_\varphi$ is relatively unramified.
\end{proposition}
\begin{proof}
The inertia subgroup of $\Gal(M/\K) \subset \W^\tau \times \Gal(\E/\K)$ is just the $\Gal(\E/\K)$ factor.  The obvious homomorphism $\Gal(\E/\K) \subset \W^\tau \times \Gal(\E/\K) \rightarrow \hat{\T} \cdot \N_{\hat{\G}}(\hat{\T})^\tau \rtimes \Gal(\E/\K)$ provides a partial splitting for the sequence
\[
1 \rightarrow \hat{\T} \rightarrow \hat{\T} \cdot \N_{\hat{\G}}(\hat{\T})^\tau \rtimes \Gal(\E/\K) \rightarrow \W^\tau \times \Gal(\E/\K) \rightarrow 1,
\]
and since its image lies within $D_\varphi$, this same map splits $D_\varphi \rightarrow \Gal(M/\K)$.
\end{proof}

\section{Regular characters from Langlands parameters} \label{sec:char_from_LP}

Suppose that $D$ is an relatively unramified group of type $L$ associated to a torus $\T$ that splits over a tame extension $M$ of $\K$.  
In this section we define a map
\[
\psi_D \colon P_\K(D, \T) \rightarrow \Hom(\T(\K)^0, \CC^\times)
\]
that will allow us to associate a character $\chi_\varphi$ to each $\varphi \in P_\K(D, \T)$.

\begin{lemma} \label{lem:character:T0_isom}
For any unramified extension $N/\K$, the norm map induces an isomorphism from $\T(N)^0_{\Gal(N/\K)}$ to $\T(\K)^0$.
\end{lemma}
\begin{proof}
Consider the Tate cohomology sequence for the $\Gal(N/\K)$-module $\T(N)^0$:
\[
0 \rightarrow \HT{-1}(\Gal(N/\K), \T(N)^0) \rightarrow \T(N)_{\Gal(N/\K)}^0 \rightarrow \T(\K)^0 \rightarrow \HT{0}(\Gal(N/\K), \T(N)^0) \rightarrow 0.
\]
Since the central map is precisely that induced by the norm, it suffices to prove that the outside two groups are trivial.

Note that
\[
T(N)^0 = \varprojlim_{r} T(N)^0 / T(N)^r.
\]
So by a result of Serre \cite{serre:LocalClassFieldThy}*{Lem. 3}, it suffices to prove that
$\HT{i}(\Gamma_n, T(N)^r / T(N)^{r+}) = 0$ for all $i$.  But $T(N)^r / T(N)^{r+}$
is connected \cite{yu:03a}*{Prop. 5.2} and thus has trivial cohomology by
Lang's Theorem \cite{lang:56a}.
\end{proof}

We can now construct $\psi_D$.  Since $M/\K$ is tame, $\Gal(M/\K_f)$ will be abelian, where $\K_f$ is the maximal unramified subextension of $M/\K$.  By Proposition \ref{prop:character:res_im_fixed}, we have a map
\[
P_\K(D, \T) \rightarrow \Hh^1(\K_f, \hat{\T})^{\Gal(\K_f/\K)}.
\]
The local Langlands correspondence for tori \cite{yu:09a}*{\S 7.5} defines an isomorphism
\[
\Hom(\T(\K_f), \CC^\times) \rightarrow \Hh^1(\K_f, \hat{\T}),
\]
and Lemma \ref{lem:character:T0_isom} gives an isomorphism
\[
\Hom(\T(\K)^0, \CC^\times) \rightarrow \Hom(\T(\K_f)^0_{\Gal(\K_f/\K)}, \CC^\times).
\]
Finally, restriction induces a homomorphism
\[
\Hom(\T(\K_f), \CC^\times) \rightarrow \Hom(\T(\K_f)^0, \CC^\times),
\]
and those characters fixed by $\Gal(\K_f/\K)$ are precisely those descending to a well defined homomorphism from the co-invariants $\T(\K_f)^0_{\Gal(\K_f/\K)}$.
Putting all of these together, we define $\psi_D$ as the composition
\begin{multline}
P_\K(D, \T) \rightarrow \Hh^1(\K_f, \hat{\T})^{\Gal(\K_f/\K)} \rightisoarrow \Hom(\T(\K_f), \CC^\times)^{\Gal(\K_f/\K)} \\
\rightarrow \Hom(\T(\K_f)^0_{\Gal(\K_f/\K)}, \CC^\times) \rightisoarrow \Hom(\T(\K)^0, \CC^\times).
\end{multline}

If $\varphi$ is a tame, discrete, regular Langlands parameter then by Proposition \ref{prop:character:Dphi_unram} $D_\varphi$ is relatively unramified.  Moreover, $\T$ splits over a tame extension of $\K$, so $\psi_{D_\varphi}$ exists.  For such a $\varphi$ we will denote the element $\psi_{D_\varphi}(\varphi) \in \Hom(\T(\K)^0, \CC^\times)$ by $\chi_\varphi$.

\subsection*{Depth of character} \label{sec:character:depth}

Just as we defined the depth of an element of $\Hh^1(\Weil_\K, \hat{\T})$ at the end of section \ref{sec:tori}, we can define the depth of an element of $P_\K(D, \T)$.  
\begin{definition}
Suppose $\varphi \in P_\K(D, \T)$.  Then the \emph{depth} of $\varphi$ is the infimum over $r \ge 0$ with 
\[
\ker(\varphi) \supset \Weil_\K^r.
\]
\end{definition}
Note that $\ker(\varphi)$ is well defined even though $\varphi \in P_\K(D, \T)$ is only defined up to conjugation by an element of $\hat{\T} \subseteq D$.

We can now generalize the depth preservation of the local Langlands correspondence for tori.
\begin{theorem}
Suppose that $\T$ splits over a tame extension $M$ of $\K$, and $D$ is an relatively unramified group of type $L$.  Then $\psi_D$ preserves depth.
\end{theorem}
\begin{proof}
We prove that each map going into the definition of $\psi_D$ preserves depth.  Let $\K_f$ be the maximal unramified subextension of $M/\K$ as in the construction of $\psi_D$.
\begin{enumerate}
\item Restricting to $\Weil_{\K_f} \subset \Weil_\K^0$ has no effect on depth since $\Weil_{\K_f}^r = \Weil_\K^r \cap \Weil_{\K_f}$.
\item The local Langlands correspondence for tori preserves depth \cite{yu:09a}*{\S 7.10}.
\item Restricting characters to $\T(\K_f)^0$ has no effect on the depth since we intersect with $\T(\K_f)^0$ in the definition of the Moy-Prasad filtration already.
\item Finally, we need to show that the norm map $\T(\K_f)^0 \rightarrow \T(\K)^0$ preserves the Moy-Prasad filtration.

Suppose first that $\T = \Res_{M/\K} \Gm$.  Then 
\[
\T(\K_f) = (M^\times)^f,
\]
and $\Gal(\K_f/\K)$ acts by permuting the coordinates.  The norm map thus just multiplies all coordinates together, which sends
\[
\T(\K_f)^r = (1 + \pi_M^r \OO_M)^f
\]
surjectively onto $\T(\K)^r = (1 + \pi_M^r \OO_M)$ for any positive integer $r$.  Since $\K_f / \K$ is unramified, we get that the Moy-Prasad filtration is preserved by the norm map on $\T$.

For a more general $\T$, we embed $\T$ into a product $\R$ of restrictions of the above form.  Since the Moy-Prasad filtration is defined as the intersection of the filtration on $\R$ with the connected N\'eron model of $\T$, our result follows from the above case and the behavior of N\'eron models under unramified base change \cite{bosch-lutkebohmert-reynaud:NeronModels}*{Prop. 10.1.3}.
\end{enumerate}
\end{proof}

\begin{corollary}
If $\varphi$ is a tame, discrete, regular Langlands parameter then $\chi_\varphi$ has depth zero.  In particular, it induces a character on the $k$-points
\[
\Tbar(\k) = \T(\K)^0 / \T(\K)^{0+}
\]
of the special fiber of the N\'eron model $\TT$.
\end{corollary}
\begin{proof}
The tameness of $\varphi$ is equivalent to $\varphi$ having depth zero.
\end{proof}

\subsection*{Regularity} \label{sec:character:regularity}

In order to prove the irreducibility of the Deligne-Lusztig representations we construct, we need to compute the stabilizer of $\chi_\varphi$ in the Weyl group of $\Tbar$.  This stabilizer will depend on the embedding of $\T$ into pure inner forms of $\G$.
\begin{lemma} \label{lem:character:Kuconj}
Let $\G'$ be pure inner form of $\G$.
\begin{enumerate}
\item There is an isomorphism $\beta \colon \G \rightarrow \G'$ defined over $\Ku$.
\item There is an element $g \in \G'(\Ku)$ with $\T(\Ku) = g\beta(\Ss(\Ku))g^{-1}$.  Moreover, the identification of $X^*(\Ss)$ with $X^*(\T)$ and $X_*(\Ss)$ with $X_*(\T)$ defined by $\beta$ and conjugation by $g$ is precisely that obtained by the construction of $\T$ as a twist of $\Ss$.
\end{enumerate}
\end{lemma}
\begin{proof}
By Steinberg's theorem \cite{serre:GaloisCohomology}*{Ch. II \S 3.3 and III \S 2.3}, $\Hh^1(\Ku, \G) = 0$, and thus all inner forms of $\G$ become isomorphic (and quasi-split) over $\Ku$.  

Since $\G$ is already quasi-split over $\K$ with totally ramified splitting field, the $\K$-rank and $\Ku$-rank of $\G$ are identical.  Since $\Ss$ contains a maximal $\K$-split torus, and because $\Ss$ and $\T$ become isomorphic over $\Ku$, they both contain a $\Ku$-split torus of dimension equal to the $\Ku$-rank of $\G$, which is the same as the $\Ku$-rank of $\G'$.  Now we note that $\G'$ has a unique conjugacy class of such maximal tori over $\Ku$ since it is quasi-split over $\Ku$.

The final statement follows from Proposition \ref{prop:ua_tori:unram_classifying_set}.
\end{proof}

Conjugation by this $g$ also takes the normalizer of $\beta(\Ss)$ to the normalizer of $\T$, and thus defines an isomorphism of the Weyl group $\W_{\Ss}$ of $\Ss$ with the Weyl group $\W_{\T}$ of $\T$, as finite group schemes over $\Ku$.  Moreover, since $\T$ splits over $M$, we can choose both the element $g$ and the isomorphism of Weyl groups to be defined over $\K_f$, the maximal unramified subextension of $M$.

Since $\Tbar$ is defined as the special fiber of the N\'eron model of $\T$, the Weyl group of $\Tbar$ is naturally identified with the sub-group scheme $\W_{\T}^{\I} \subset \W_{\T}$.  As our isomorphism $\W_{\Ss} \rightisoarrow \W_{\T}$ is defined over $\Ku$, we may identify $\W_{\Ss}^{\I}$ and $\W_{\T}^{\I}$.  We define a character $\chi_\varphi' \colon \Ss(\K_f) \rightarrow \CC^\times$ by pulling $\chi_\varphi$ back to $\T(\K_f)$ using the norm (essentially moving back one step in the application of $\psi_{D_\varphi}$) and then conjugating $\T(\K_f)$ to $\Ss(\K_f)$ with $g$.  Since our identification of $\W_{\Ss}$ with $\W_{\T}$ is done via conjugation by $g$ as well, an element of $\W_{\Ss}^{\I}$ will fix $\chi_\varphi'$ if and only if the corresponding element of $\W_{\T}^{\I}$ fixes $\chi_\varphi$.  For the rest of this section we will work with $\W_{\Ss}$; write $\W$ for $\W_{\Ss}$.

In order to state the result on the stabilizer of $\chi_\varphi'$ in $\W_{\Ss}^{\I}$, we need to recall some notation from Reeder \cite{reeder:10a}.  Set $Y = X^*(\Ss)$ and $Y_\RR = Y \otimes \RR$.  Any element $\vartheta \in \Gal(E/\K)$ acts via a pinned automorphism on $\hat{\G}$.  Suppose that $\vartheta$ has order $m$, and let
\[
P_\vartheta = m^{-1}(1 + \vartheta + \cdots + \vartheta^{m-1}) \in \End(Y_\RR).
\]
Set
\[
Y_\vartheta = P_\vartheta Y,
\]
the projection of $Y$ onto $Y_\RR^\vartheta$.  We then define
\[
\widetilde{\W}_\vartheta = \W^\vartheta \ltimes Y_\vartheta.
\]
We have the exact sequence
\[
1 \rightarrow Y \rightarrow Y \otimes \CC \xrightarrow{\exp} \hat{\Ss} \rightarrow 1,
\]
and the subspace $Y_\RR^\vartheta$ maps under $\exp$ into $\hat{\Ss}^\vartheta$.  By \cite{reeder:10a}*{Lem. 3.4}, if $x, x' \in Y_\RR^\vartheta$, then the elements $\exp(x)\vartheta$ and $\exp(x')\vartheta$ of ${}^L\G$ are $\hat{\Ss}$-conjugate if and only if $x-x' \in Y_\vartheta$.  We will apply this result to the case that $\vartheta = \tau$, and note that $\W_{\Ss}^\I = \W_{\Ss}^\tau$.  Our next goal is to define an alcove $C_\tau$ in $Y_\RR^\tau$.  

Reeder denotes by $\Phi / \vartheta$ the set of $\vartheta$-equivalence classes of roots in $\Phi(\hat{\G}, \hat{\Ss})$ and for each $a \in \Phi / \vartheta$, he sets
\[
\gamma_a = \sum_{\alpha \in a} \bar{\alpha} \qquad \mbox{and} \qquad \Phi_\vartheta = \{\gamma_a \st a \in \Phi / \vartheta\},
\]
where $\bar{\alpha}$ is the restriction of the root $\alpha$ to $W^\vartheta$.

He defines $I_\vartheta$ as the set of orbits in $\{1, \ldots, l\}$ under the permutation induced by the action of $\vartheta$ on the set $\{\alpha_1, \ldots, \alpha_l\}$ of simple roots in $\Delta(\hat{\G}, \hat{\B})$, for $\iota \in I_\vartheta$ sets $a_\iota \in \Phi / \vartheta$ as the equivalence class containing $\{a_i \st i \in \iota\}$, and defines $\gamma_\iota = \gamma_{a_\iota}$.  The set $\Delta_\vartheta = \{\gamma_\iota \st \iota \in I_\vartheta\}$ is a base for the reduced root system $\Phi_\vartheta$, and he can thus define $\tilde{\gamma}_0$ as the highest root of $\Phi_\vartheta$ with respect to $\Delta_\vartheta$.  He then sets
\[
\tilde{I}_\vartheta = \{0\} \cup I_\vartheta, \qquad \mbox{and} \qquad \gamma_0 = 1 - \tilde{\gamma}_0.
\]

We can now define an alcove $C_\vartheta$ in $W^\vartheta$ by
\[
C_\vartheta = \{x \in Y_\RR^\vartheta \st \gamma_\iota > 0 \ \ \forall \iota \in \tilde{I}_\vartheta\}.
\]
There is a unique element $y$ in the closure of $C_\tau$ satisfying
\[
\varphi(\ttau) = \exp(y)\tau.
\]
Finally, we let 
\[
\W_{\varphi(\ttau)} = \N_{\hat{\G}}(\hat{\Ss})^{\varphi(\ttau)} / \hat{\Ss}^\tau
\]
be the subgroup of elements of $\W^\tau$ representable by a $\varphi(\ttau)$ fixed element of $\N_{\hat{\G}}(\hat{\Ss})$.

From the proof of \cite{reeder:10a}*{Lem. 3.9}, the projection $\widetilde{\W}_\tau \rightarrow \W^\tau$ maps the stabilizer $\widetilde{\W}_{\tau, y}$ of $y$ in $\widetilde{\W}_\tau$ isomorphically onto $\W_{\varphi(\ttau)}$.

\begin{proposition} \label{prop:character:Wfix}
\[
\{ w \in \W^\tau \st w \cdot \chi_\varphi' = \chi_\varphi'\} \subseteq \W_{\varphi(\ttau)}.
\]
\end{proposition}
\begin{proof}
The local Langlands correspondence for tori is given by the following series of isomorphisms \cite{yu:09a}*{\S 7.7}:
\begin{align*}
\Hom(\T(K_f), \CC^\times) &\rightisoarrow \Hom(\T(M), \CC^\times)_{\Gal(M/K_f)} \\
&\rightisoarrow \Hom(M^\times \otimes_\ZZ X_*(\T), \CC^\times)_{\Gal(M/K_f)} \\
&\rightisoarrow \Hom(M^\times, X^*(\T) \otimes_\ZZ \CC^\times)_{\Gal(M/K_f)} \\
&\rightisoarrow \Hh^1(\Weil_M, \hat{\T})_{\Gal(M/K_f)} \\
&\rightisoarrow \Hh^1(\Weil_{K_f}, \hat{\T}).
\end{align*}

We can translate to $\Ss$ by conjugating by $g$, and then trace through the action of $\W^\tau$.  The action of $\W^\tau$ on $\Ss(K_f)$ comes from its action on $X_*(\Ss)$, and this corresponds to the standard action of $\W^\tau$ on $\hat{\Ss}$.  Since $\chi_\varphi$ maps to $\varphi$, an element $w \in \W^\tau$ will fix $\chi_\varphi$ if and only if it fixes the restriction of $\varphi$ to $\Weil_{\K_f}$.  Note that $\varphi \in \Hh^1(\Weil_{\K_f}, \hat{\T})$ is determined by $\varphi(\ttau)$ and $\varphi(\Fr^f)$.  The condition that $\varphi$ is fixed by $w$ translates to the requirement that $w \cdot \varphi(\ttau)$ is conjugate to $\varphi(\ttau)$ by some $t \in \hat{\T}$, and that $w \cdot \varphi(\Fr^f)$ is also conjugate to $\varphi(\Fr^f)$ by the same $t$.

We now invoke \cite{reeder:10a}*{Lem. 3.4} to replace the condition that $w \cdot \varphi(\ttau) = \exp(w \cdot y)\tau$ be conjugate to $\varphi(\ttau) = \exp(y)\tau$ with the requirement that
\[
y - w\cdot y \in Y_\tau.
\]

Since $\widetilde{\W}_\tau = \W^\tau \ltimes Y_\tau$, the statement that we can translate from $y$ to $w\cdot y$ by an element of $Y_\tau$ is equivalent to the statement that $y$ is be fixed by some element of $\widetilde{\W}_\tau$.  The image of this element under the projection $\widetilde{\W}_\tau \rightarrow \W^\tau$ gives us a $w \in \W^\tau$ so that $w \cdot \varphi(\ttau)$ is conjugate to $\varphi$.
Therefore any $w$ fixing $\chi_\varphi'$ must be in the image of $\Wexaff_{\tau, y}$ under the projection $\Wexaff_\tau \rightarrow \W^\tau$, which is precisely $\W_{\varphi(\ttau)}$.
\end{proof}

\section{Supercuspidal representations from Langlands parameters} \label{sec:induction:rep_const}

We may now define a complex admissible representation $\pi = \pi_{\varphi, \uk}$ of $G$ in a sequence of steps.

\begin{enumerate}
\item Since the character $\chi_\varphi$ has depth zero, it descends to a character on $\Tbar$.  Together with the torus $\Tbar \subseteq \Gbar^\circ$, this character defines a Deligne-Lusztig representation $\pi^\circ$ of $\Gbar^\circ$.
\item We obtain a representation of the parahoric subgroup $G^\circ$ via the reduction map $G^\circ \rightarrow \Gbar^\circ$; we will also call this representation $\pi^\circ$.
\item Define a representation $\pi^\flat$ on the maximal compact subgroup $G^\flat$ by a finite induction from $G^\circ$.
\item Finally, define a representation $\pi$ on all of $G$ by compact induction from $G^\flat$.
\end{enumerate}

In this section we elaborate on the different steps in this process and give conditions under which the representation at each step is irreducible.  

\subsection*{Representation of the parahoric}
The representation $\pi^\circ$ will be irreducible if and only if the only $\Fr$-invariant of the Weyl group of $\Tbar$ fixing $\chi_\varphi$ is the identity \cite{carter:93a}*{Thm. 7.3.4}, namely that $\chi_\varphi$ is in \emph{general position}.
\begin{proposition}
Suppose that $\Z_{\hat{\G}}(\varphi(\ttau)) = \hat{\Ss}^{\tau}$.  Then $\chi_\varphi$ is in general position.
\end{proposition}
\begin{proof}
We use the notation of Section \ref{sec:character:regularity}.  Note that $\W_{\varphi(\ttau)}$ is trivial, since any nontrivial element would lift to an element of $\hat{\G}$ centralizing $\varphi(\ttau)$ but lying outside $\hat{\Ss}^\tau$, contradicting the assumption on $\varphi(\ttau)$.  The result now follows by Proposition \ref{prop:character:Wfix}.
\end{proof}

Inflation of $\pi^\circ$ to $G^\circ$ does not affect its irreducibility.
\subsection{Induction to the normalizer}
The induction of $\pi^\circ$ from $G^\circ$ to $G^\flat$ does not always remain irreducible, but we may pick out an irreducible factor using the central character when it does not.

We have two methods for obtaining a character on the center $Z = \Zz(\K)$.  Since $Z$ is compact and central, $Z \subset G^\flat$.  We can thus restrict $\pi^\circ$ to get a character $\epsilon$ from $Z^\circ$ to the center of a general linear group, which is isomorphic to $\CC^\times$.  On the other hand, Gross and Reeder give a recipe for the central character $\omega_\varphi \colon Z \rightarrow \CC^\times$ \cite{gross-reeder:09a}*{\S 8}.
\begin{lemma} \label{lem:induction:Z0agree}
The two characters $\epsilon$ and $\omega_\varphi$ agree on $Z^\circ$.  
\end{lemma}
\begin{proof}
From the description of the Deligne-Lusztig representation in Carter \cite{carter:93a}*{\S 7.2}, we see that central elements $z \in \Gbar$ scale by $\chi_\varphi(z)$.  Recall the description of $\omega_\varphi$ as the image of $\varphi$ under the composition
\[
\Hh^1(\K, \hat{\G}) \rightarrow \Hh^1(\K, \hat{\Zz}) \rightarrow \Hom(Z, \CC^\times).
\]
If we instead proceed by restricting to $\Hh^1(\K_f, \hat{\G})$, projecting onto $\Hh^1(\K_f, \hat{\Zz})$, mapping to $\Hom(Z, \CC^\times)$ and then restricting to $\Hom(Z^\circ, \CC^\times)$ we will get the same character since the following diagrams commute:
\begin{enumerate}
\item
\[
\begin{tikzpicture}[font=\normalsize]
\matrix [matrix of math nodes, row sep=1cm, column sep=1.5cm]
{
|(HT)| \Hh^1(\K_f, \hat{\T}) & |(HoT)| \Hom(\T(\K_f), \CC^\times) \\
|(HZ)| \Hh^1(\K_f, \hat{\Zz}) & |(HoZ)| \Hom(\Zz(\K_f), \CC^\times) \\
};
\draw[->] (HT) -- (HoT);
\draw[->] (HT) -- (HZ);
\draw[->] (HZ) -- (HoZ);
\draw[->] (HoT) -- (HoZ);
\end{tikzpicture}
\]
where the horizontal maps are the local Langlands correspondence, the left map is induced by the quotient map $\hat{\G} \rightarrow \hat{\Zz}$ and the right map is restriction.
\item 
\[
\begin{tikzpicture}[font=\normalsize]
\matrix [matrix of math nodes, row sep=1cm, column sep=1.5cm]
{
|(H)| \Hh^1(\K, \hat{\Zz}) & |(Ho)| \Hom(\Zz(\K), \CC^\times) \\
|(Hf)| \Hh^1(\K_f, \hat{\Zz}) & |(Hof)| \Hom(\Zz(\K_f), \CC^\times) \\
};
\draw[->] (H) -- (Ho);
\draw[->] (H) -- (Hf);
\draw[->] (Ho) -- (Hof);
\draw[->] (Hf) -- (Hof);
\end{tikzpicture}
\]
where the left map is restriction and the right is induced by $\Nm : \Zz(\K_f) \rightarrow \Zz(\K)$.
\end{enumerate}
\end{proof}

With this lemma in hand, we can define the representation $\pi^\flat$ of $G^\flat$.
\begin{proposition}
There is a unique irreducible representation $\pi^\flat$ of $G^\flat$ satisfying:
\begin{enumerate}
\item $\pi^\flat$ is a sub-representation of the induction $\Ind^{G^\flat}_{G^\circ} \pi^\circ$,
\item the restriction of $\pi^\flat$ to $Z$ agrees with $\omega_\Zz$.
\end{enumerate}
\end{proposition}
\begin{proof}
There are two cases.  Assume first that $\G$ is an even unitary group.  Then the induction $\Ind^{G^\flat}_{G^\circ} \pi^\circ$ remains irreducible \cite{isaacs:CharacterTheoryOfFiniteGroups}*{Thm. 6.11}.  By Lemma \ref{lem:induction:Z0agree} and Proposition \ref{prop:induction:connected_center}, the induction satisfies the second property.  So we may set

\[
\pi^\flat = \Ind^{G^\flat}_{G^\circ} \pi^\circ.
\]

For odd unitary groups, the induction has two irreducible sub-representations.  By Proposition \ref{prop:induction:connected_center}, there is a central element $z$ lying in the nontrivial coset of $G^\circ \subset G^\flat$.  The two irreducible representations in the induction will take different values on $z$, and thus exactly one of these will satisfy the second requirement.  On the other hand, Lemma \ref{lem:induction:Z0agree} guarantees that both pieces of the induction will agree with $\omega_\Zz$ on $Z^\circ$, so our chosen sub-representation will satisfy both desired properties.
\end{proof}
\subsection*{Compact induction}
The representation $\pi^\flat$ acts on a finite dimensional $\CC$-vector space.  We now define
\[
\pi = \ind_{G^\flat}^G \pi^\flat.
\]
By a theorem of Moy and Prasad, $\pi$ is irreducible and supercuspidal \cite{moy-prasad:96a}*{\S 6.3, Prop. 6.6}.

\subsection*{\texorpdfstring{$L$-packets}{L-packets}} \label{sec:induction:Lpacketsize}

Let $\G = \U(V)$ as normal, and set $n=\dim V$.  The $L$-packet $\Pi_\varphi$ associated to a tame, discrete, regular Langlands parameter $\varphi$ for $\G$ consists of the representations $\pi$ constructed in the previous section, parameterized by the embeddings $\rho \colon \T \hookrightarrow \G'$ as $\G'$ ranges over the pure inner forms of $\G$.  The parameter $\varphi$ controls the size of $\Pi_\varphi$ as follows.

\begin{proposition}
Let $j$ be the number of cycles in the permutation obtained by projecting $\varphi(\Fr)$ onto the Weyl group of the unique maximal torus containing $\varphi(\ttau)$.  Then there are $2^j$ representations in $\Pi_\varphi$.
\end{proposition}
\begin{proof}
Note that the action of Frobenius on $\T$ is defined by $\varphi(\Fr)$, and the decomposition in the proof of Theorem \ref{thm:ua_tori:elemental_decomposition} is determined by the cycles in that action on $X^*(\T)$.  When $n$ is even, $j$ is the number of tori $\T_{s_i}$ in the decomposition of Theorem \ref{thm:ua_tori:elemental_decomposition}.  For odd $n$, the $\U_1$ factor adds $1$ to this total.

For each $\T_{s_i}$, the representation $\pi$ associated to the embedding determined by the tuple $\uk$ depends only on the choice of $\kappa_i$ modulo $\Nm_{E_i/L_i} E_i^\times$.  Since $L_i^\times / \Nm_{E_i/L_i} E_i^\times$ has order 2, the result follows.
\end{proof}

Write $\omega$ for the image of $\varphi(\Fr)$ in $\W^\I$.  When $V$ has dimension $2m$, the smallest $L$-packets, of cardinality 2, occur when $\omega$ is a Coxeter element.  The largest, of size $2^m$, occur when $\omega$ is a product of $m$ commuting transpositions.  For any $L$-packet, each embedding determines a vertex of $\BB(\G)$ stabilized by the image of $\T(\K)$ in $\G'(\K)$.  Up to conjugacy within $\G'(\K)$ each embedding is determined by the choice of even or odd valuation for each $\kappa_i$, and one can pick out the type of the stabilized vertex in the tables of Figure \ref{fig:background:Ureductions} using Theorem \ref{thm:induction:Ured} and Proposition \ref{prop:induction:ukform}.

Similar results hold when $V$ has dimension $2m+1$: the smallest $L$-packets have cardinality 4 and occur when $\omega$ is a Coxeter element.  The largest have size $2^m$.

\appendix
\section{The Local Index} \label{app:loc_indx}

Tits describes \cite{tits:79a}*{\S 1.11} the local index, a visual tool used to classify simple groups over $\K$ \cite{tits:79a}*{\S 4}, to understand the geometry of the apartment \cite{tits:79a}*{\S 1.8} and to find the reduction of models associated to vertices in the building.  In this appendix we reproduce the local indices of unitary groups.  The local index of $\G$ is:
\begin{enumerate}
\item the extended Dynkin diagram $\DynKu$ of $\G$ over $\Ku$,
\item the action of $\GalKuK$ on $\DynKu$.
\end{enumerate}
This data determines the extended Dynkin diagram $\Dyn$ of $\G$ over $\K$ according to an algorithm described by Tits.  In particular, there is a bijection between vertices $v$ of $\Dyn$ and orbits $O(v)$ for the action of $\GalKuK$ on $\DynKu$.  In Figure \ref{fig:background:Ureductions} we give the local indices for unitary groups associated to both unramified and tamely ramified $E/\K$.  The lower diagram is $\Dyn$ and the upper is $\DynKu$; the vertices of $O(v)$ are placed vertically above $v$.  In the case that $\GalKuK$ acts trivially on $\DynKu$, the two diagrams are the same and the upper is omitted.  The hyperspecial vertices are denoted \begin{tikzpicture} \node [hsvert] {}; \end{tikzpicture} and the other special vertices are denoted \begin{tikzpicture} \node [svert] {}; \end{tikzpicture}.  Thick lines are used when one simple root is a negative multiple of the other.  As normal, arrows point toward the shorter root if there is a difference in length.  

{\footnotesize
\begin{figure}
\begin{tabular}{cccccc}

\phantom{hihihello}&\textbf{$E / \K$ unramified} &\phantom{middle1}&\phantom{middle2}& \textbf{$E/\K$ ramified} &\phantom{byebyebye} \\
\noalign{\vspace{.4cm}}
\multicolumn{6}{c}{quasi-split $\U_3(E/\K)$} \\
\multicolumn{6}{c}{\begin{tikzpicture}[x=.5cm,y=.25cm]
\matrix[row sep=.4cm,column sep=3cm, ampersand replacement=\&]
{
\node (ua0) at (-1,0) [uvert] {};
\node (ua1) at (1,1) [uvert] {};
\node (ua2) at (1,-1) [uvert] {};
\draw (ua0) -- (ua1) .. controls (1.6,1.3) and (1.6,-1.3) .. (ua2) -- (ua0); \& \\
\node (a0) at (-1,0) [hsvert] {};
\node (a0L) at (-1.16,-1) [anchor=north east] {\phantom{$\U_1\times$}$\U_3$};
\draw [Lline] (a0L.north east) -- (a0);
\node (a1) at (1,0) [sxvert] {};
\node (a1L) at (1.16,-1) [anchor=north west] {$\U_1 \times \U_2$};
\draw [Lline] (a1L.north west) -- (a1);
\draw [tline] (a0) -- (a1); \& 
\node (b0) at (-1,0) [svert] {};
\node (b0L) at (-1.16,-1) [anchor=north east] {\phantom{$\Sp_2\times$}$\Orth_3$};
\draw [Lline] (b0L.north east) -- (b0);
\node (b1) at (1,0) [svert] {};
\node (b1L) at (1.16,-1) [anchor=north west] {$\Orth_1 \times \Sp_2$};
\draw [Lline] (b1L.north west) -- (b1);
\draw [tline] (b0) -- (b1); \\
};
\end{tikzpicture}}\\
\multicolumn{6}{c}{quasi-split $\U_{2m+1}(E/\K)$} \\
\multicolumn{6}{c}{
\begin{tikzpicture}[x=.5cm,y=.25cm]
\matrix[row sep=.4cm,column sep=0cm,ampersand replacement=\&]
{
\node (uc0) at (-5,0) [uvert] {};
\node (uc1) at (-3,1) [uvert] {};
\node (uc2) at (-1,1) [uvert] {};
\node (uc3) at (1,1) [uvert] {};
\node (uc4) at (3,1) [uvert] {};
\node (uc5a) at (5,1) [uvert] {};
\node (uc5b) at (5,-1) [uvert] {};
\node (uc6) at (3,-1) [uvert] {};
\node (uc7) at (1,-1) [uvert] {};
\node (uc8) at (-1,-1) [uvert] {};
\node (uc9) at (-3,-1) [uvert] {};
\draw (uc2) -- (uc1) -- (uc0) -- (uc9) -- (uc8) (uc3) -- (uc4) -- (uc5a) .. controls (5.6, 1) and (5.6, -1) .. (uc5b) -- (uc6) -- (uc7);
\draw [dashline] (uc2) -- (uc3) (uc7) -- (uc8); \& \\
\node (c0) at (-5,0) [hsvert] {};
\node (c0L) at (-5.83,-5) {$\U_{2m+1}$} edge [Lline] (c0);
\node (c1) at (-3,0) [lvert] {};
\node (c1L) at (-3.42,-2.5) {$\U_{2m-1}\times\U_2$} edge [Lline] (c1);
\node (c2) at (-1,0) [lvert] {};
\node (c2L) at (-1.83,-5) {$\U_{2m-3}\times\U_4$} edge [Lline] (c2);
\node (c3) at (1,0) [lvert] {};
\node (c3L) at (.58,-2.5) {$\U_5\times\U_{2m-4}$} edge [Lline] (c3);
\node (c4) at (3,0) [lvert] {};
\node (c4L) at (2.17,-5) {$\U_3\times\U_{2m-2}$} edge [Lline] (c4);
\node (c5) at (5,0) [sxvert] {};
\node (c5L) at (4.58,-2.5) {$\U_1\times\U_{2m}$} edge [Lline] (c5);
\draw [dline] (c0) -- (c1);
\draw (c1) -- (c2) (c3) -- (c4);
\draw [dashline] (c2) -- (c3);
\draw [dline] (c4) -- (c5); \&
\node (d0) at (-5,0) [svert] {};
\node (d0L) at (-5.83,-5) {$\Orth_{2m+1}$} edge [Lline] (d0);
\node (d1) at (-3,0) [lvert] {};
\node (d1L) at (-3.42,-2.5) {$\Orth_{2m-1}\times\Sp_2$} edge [Lline] (d1);
\node (d2) at (-1,0) [lvert] {};
\node (d2L) at (-1.83,-5) {$\Orth_{2m-3}\times\Sp_4$} edge [Lline] (d2);
\node (d3) at (1,0) [lvert] {};
\node (d3L) at (.58,-2.5) {$\Orth_5\times\Sp_{2m-4}$} edge [Lline] (d3);
\node (d4) at (3,0) [lvert] {};
\node (d4L) at (2.17,-5) {$\Orth_3\times\Sp_{2m-2}$} edge [Lline] (d4);
\node (d5) at (5,0) [svert] {};
\node (d5L) at (4.58,-2.5) {$\Orth_1\times\Sp_{2m}$} edge [Lline] (d5);
\draw [dline] (d0) -- (d1);
\draw (d1) -- (d2) (d3) -- (d4);
\draw [dashline] (d2) -- (d3);
\draw [dline] (d4) -- (d5); \\
};
\end{tikzpicture}}\\
\multicolumn{6}{c}{quasi-split $\U_4(E/\K)$} \\
\multicolumn{6}{c}{
\begin{tikzpicture}[x=.5cm,y=.25cm]
\matrix[row sep=.4cm,column sep=3cm,ampersand replacement=\&]
{
\node (ue0) at (-2,0) [uvert] {};
\node (ue1) at (0,1) [uvert] {};
\node (ue2) at (2,0) [uvert] {};
\node (ue3) at (0,-1) [uvert] {};
\draw (ue0) -- (ue1) -- (ue2) -- (ue3) -- (ue0); \& \\
\node (e0) at (-2,0) [hsvert] {};
\node (e0L) at (-2.16,-1) [anchor=north east] {$\U_4$};
\draw [Lline] (e0L.north east) -- (e0);
\node (e1) at (0,0) [lvert] {};
\node (e1L) at (0,-1) [anchor=north] {$\U_2 \times \U_2$};
\draw [Lline] (e1L.north) -- (e1);
\node (e2) at (2,0) [hsvert] {};
\node (e2L) at (2.16,-1) [anchor=north west] {$\U_4$};
\draw [Lline] (e2L.north west) -- (e2);
\draw [dline] (e0) -- (e1);
\draw [dline] (e2) -- (e1); \&
\node (f0) at (-2,0) [svert] {};
\node (f0L) at (-2.16,-1) [anchor=north east] {$\Sp_4$};
\draw [Lline] (f0L.north east) -- (f0);
\node (f1) at (0,0) [lvert] {};
\node (f1L) at (0,-1) [anchor=north] {$\Orth_4$};
\draw [Lline] (f1L.north) -- (f1);
\node (f2) at (2,0) [svert] {};
\node (f2L) at (2.16,-1) [anchor=north west] {$\Sp_4$};
\draw [Lline] (f2L.north west) -- (f2);
\draw [dline] (f1) -- (f0);
\draw [dline] (f1) -- (f2); \\
};
\end{tikzpicture}}\\
\multicolumn{6}{c}{quasi-split $\U_{2m}(E/\K)$} \\
\multicolumn{6}{c}{
\begin{tikzpicture}[x=.5cm,y=.25cm]
\matrix[row sep=-.125cm,column sep=0cm,ampersand replacement=\&]
{
\node (ug0) at (-5,0) [uvert] {};
\node (ug1) at (-3,1) [uvert] {};
\node (ug2) at (-1,1) [uvert] {};
\node (ug3) at (1,1) [uvert] {};
\node (ug4) at (3,1) [uvert] {};
\node (ug5) at (5,0) [uvert] {};
\node (ug6) at (3,-1) [uvert] {};
\node (ug7) at (1,-1) [uvert] {};
\node (ug8) at (-1,-1) [uvert] {};
\node (ug9) at (-3,-1) [uvert] {};
\draw (ug2) -- (ug1) -- (ug0) -- (ug9) -- (ug8) (ug3) -- (ug4) -- (ug5) -- (ug6) -- (ug7);
\draw [dashline] (ug2) -- (ug3) (ug7) -- (ug8); \& \\
\node (g0) at (-5,0) [hsvert] {};
\node (g0L) at (-5.83,-5) {$\U_{2m}$} edge [Lline] (g0);
\node (g1) at (-3,0) [lvert] {};
\node (g1L) at (-3.42,-2.5) {$\U_{2m-2}\times\U_2$} edge [Lline] (g1);
\node (g2) at (-1,0) [lvert] {};
\node (g2L) at (-1.83,-5) {$\U_{2m-4}\times\U_4$} edge [Lline] (g2);
\node (g3) at (1,0) [lvert] {};
\node (g3L) at (.58,-2.5) {$\U_4\times\U_{2m-4}$} edge [Lline] (g3);
\node (g4) at (3,0) [lvert] {};
\node (g4L) at (2.17,-5) {$\U_2\times\U_{2m-2}$} edge [Lline] (g4);
\node (g5) at (5,0) [hsvert] {};
\node (g5L) at (4.58,-2.5) {$\U_{2m}$} edge [Lline] (g5);
\draw [dline] (g0) -- (g1);
\draw (g1) -- (g2) (g3) -- (g4);
\draw [dashline] (g2) -- (g3);
\draw [dline] (g5) -- (g4); \&
\node (h0) at (-5,0) [lvert] {};
\node (h0L) at (-5.83,-5) {$\Orth_{2m}$} edge [Lline] (h0);
\node (h1) at (-3,0) [lvert] {};
\node (h1L) at (-3.42,-2.5) {$\Orth_{2m-2}\times\Sp_2$} edge [Lline] (h1);
\node (h2) at (-1,0) [lvert] {};
\node (h2L) at (-1.83,-5) {$\Orth_{2m-4}\times\Sp_4$} edge [Lline] (h2);
\node (h3) at (1,0) [lvert] {};
\node (h3L) at (.42,-2.5) {$\Orth_6\times\Sp_{2m-6}$} edge [Lline] (h3);
\node (h4) at (3,0) [lvert] {};
\node (h4L) at (2.17,-5) {$\Orth_4\times\Sp_{2m-4}$} edge [Lline] (h4);
\node (h5) at (5,2) [svert] {};
\node (h6) at (5,-2) [svert] {};
\node (h5L) at (6.5,0) {$\Sp_{2m}$} edge [Lline] (h5) edge [Lline] (h6);
\draw [dline] (h0) -- (h1);
\draw [dashline] (h2) -- (h3);
\draw (h1) -- (h2) (h3) -- (h4) -- (h5) -- (h4) -- (h6); \\
};
\end{tikzpicture}}\\
\multicolumn{6}{c}{non-quasi-split $\U_4(E/\K)$} \\
\multicolumn{6}{c}{
\begin{tikzpicture}[x=.5cm,y=.25cm]
\matrix[row sep=.4cm,column sep=3cm,ampersand replacement=\&]
{
\node (ui0) at (-1,1) [uvert] {};
\node (ui1) at (1,1) [uvert] {};
\node (ui2) at (1,-1) [uvert] {};
\node (ui3) at (-1,-1) [uvert] {};
\draw (ui0) -- (ui1) .. controls (1.6,1) and (1.6,-1) .. (ui2) -- (ui3) .. controls (-1.6,-1) and (-1.6,1) .. (ui0); \& 
\node (uj0) at (-1,0) [uvert] {};
\node (uj1) at (1,1) [uvert] {};
\node (uj2) at (1,-1) [uvert] {};
\draw [dline] (uj0) -- (uj1);
\draw [dline] (uj0) -- (uj2); \\
\node (i0) at (-1,0) [sxvert] {};
\node (i0L) at (-1.16,-1) [anchor=north east] {$\U_3\times\U_1$};
\draw [Lline] (i0L.north east) -- (i0);
\node (i1) at (1,0) [sxvert] {};
\node (i1L) at (1.16,-1) [anchor=north west] {$\U_1 \times \U_3$};
\draw [Lline] (i1L.north west) -- (i1);
\draw [ttline] (i0) -- (i1); \& 
\node (j0) at (-1,0) [svert] {};
\node (j0L) at (-1.16,-1) [anchor=north east] {\phantom{$\Sp_2\times$}$\Orth_4'$};
\draw [Lline] (j0L.north east) -- (j0);
\node (j1) at (1,0) [svert] {};
\node (j1L) at (1.16,-1) [anchor=north west] {$\Orth_2' \times \Sp_2$};
\draw [Lline] (j1L.north west) -- (j1);
\draw [tline] (j0) -- (j1); \\
};
\end{tikzpicture}}\\
\multicolumn{6}{c}{non-quasi-split $\U_{2m}(E/\K)$} \\
\multicolumn{6}{c}{
\begin{tikzpicture}[x=.5cm,y=.25cm]
\matrix[row sep=.4cm,column sep=0cm,ampersand replacement=\&]
{
\node (uk0a) at (-5,1) [uvert] {};
\node (uk1) at (-3,1) [uvert] {};
\node (uk2) at (-1,1) [uvert] {};
\node (uk3) at (1,1) [uvert] {};
\node (uk4) at (3,1) [uvert] {};
\node (uk5a) at (5,1) [uvert] {};
\node (uk5b) at (5,-1) [uvert] {};
\node (uk6) at (3,-1) [uvert] {};
\node (uk7) at (1,-1) [uvert] {};
\node (uk8) at (-1,-1) [uvert] {};
\node (uk9) at (-3,-1) [uvert] {};
\node (uk0b) at (-5,-1) [uvert] {};
\draw (uk2) -- (uk1) -- (uk0a) .. controls (-5.6,1) and (-5.6,-1) .. (uk0b) -- (uk9) -- (uk8) (uk3) -- (uk4) -- (uk5a) .. controls (5.6,1) and (5.6,-1) .. (uk5b) -- (uk6) -- (uk7);
\draw [dashline] (uk2) -- (uk3) (uk7) -- (uk8); \& 
\node (ul0) at (-5,0) [uvert] {};
\node (ul1) at (-3,0) [uvert] {};
\node (ul2) at (-1,0) [uvert] {};
\node (ul3) at (1,0) [uvert] {};
\node (ul4) at (3,0) [uvert] {};
\node (ul5) at (5,1) [uvert] {};
\node (ul6) at (5,-1) [uvert] {};
\draw [dline] (ul0) -- (ul1);
\draw (ul1) -- (ul2) (ul3) -- (ul4) -- (ul5) -- (ul4) -- (ul6);
\draw [dashline] (ul2) -- (ul3); \\
\node (k0) at (-5,0) [svert] {};
\node (k0L) at (-5.83,-5) {$\U_{2m-1}\times\U_1$} edge [Lline] (k0);
\node (k1) at (-3,0) [lvert] {};
\node (k1L) at (-3.42,-2.5) {$\U_{2m-3}\times\U_3$} edge [Lline] (k1);
\node (k2) at (-1,0) [lvert] {};
\node (k2L) at (-1.83,-5) {$\U_{2m-5}\times\U_5$} edge [Lline] (k2);
\node (k3) at (1,0) [lvert] {};
\node (k3L) at (.58,-2.5) {$\U_5\times\U_{2m-5}$} edge [Lline] (k3);
\node (k4) at (3,0) [lvert] {};
\node (k4L) at (2.17,-5) {$\U_3\times\U_{2m-3}$} edge [Lline] (k4);
\node (k5) at (5,0) [svert] {};
\node (k5L) at (4.58,-2.5) {$\U_1\times\U_{2m-1}$} edge [Lline] (k5);
\draw [dline] (k1) -- (k0);
\draw (k1) -- (k2) (k3) -- (k4);
\draw [dashline] (k2) -- (k3);
\draw [dline] (k4) -- (k5); \& 
\node (l0) at (-5,0) [svert] {};
\node (l0L) at (-5.83,-5) {$\Orth_{2m}'$} edge [Lline] (l0);
\node (l1) at (-3,0) [lvert] {};
\node (l1L) at (-3.42,-2.5) {$\Orth_{2m-2}'\times\Sp_2$} edge [Lline] (l1);
\node (l2) at (-1,0) [lvert] {};
\node (l2L) at (-1.83,-5) {$\Orth_{2m-4}'\times\Sp_4$} edge [Lline] (l2);
\node (l3) at (1,0) [lvert] {};
\node (l3L) at (.58,-2.5) {$\Orth_6'\times\Sp_{2m-6}$} edge [Lline] (l3);
\node (l4) at (3,0) [lvert] {};
\node (l4L) at (2.17,-5) {$\Orth_4'\times\Sp_{2m-4}$} edge [Lline] (l4);
\node (l5) at (5,0) [svert] {};
\node (l5L) at (4.58,-2.5) {$\Orth_2'\times\Sp_{2m-2}$} edge [Lline] (l5);
\draw [dline] (l0) -- (l1);
\draw (l1) -- (l2) (l3) -- (l4);
\draw [dashline] (l2) -- (l3);
\draw [dline] (l4) -- (l5); \\
};
\end{tikzpicture}}\\
\end{tabular}
\caption{Local Indices and Reductions of Unitary Groups} \label{fig:background:Ureductions}
\end{figure}
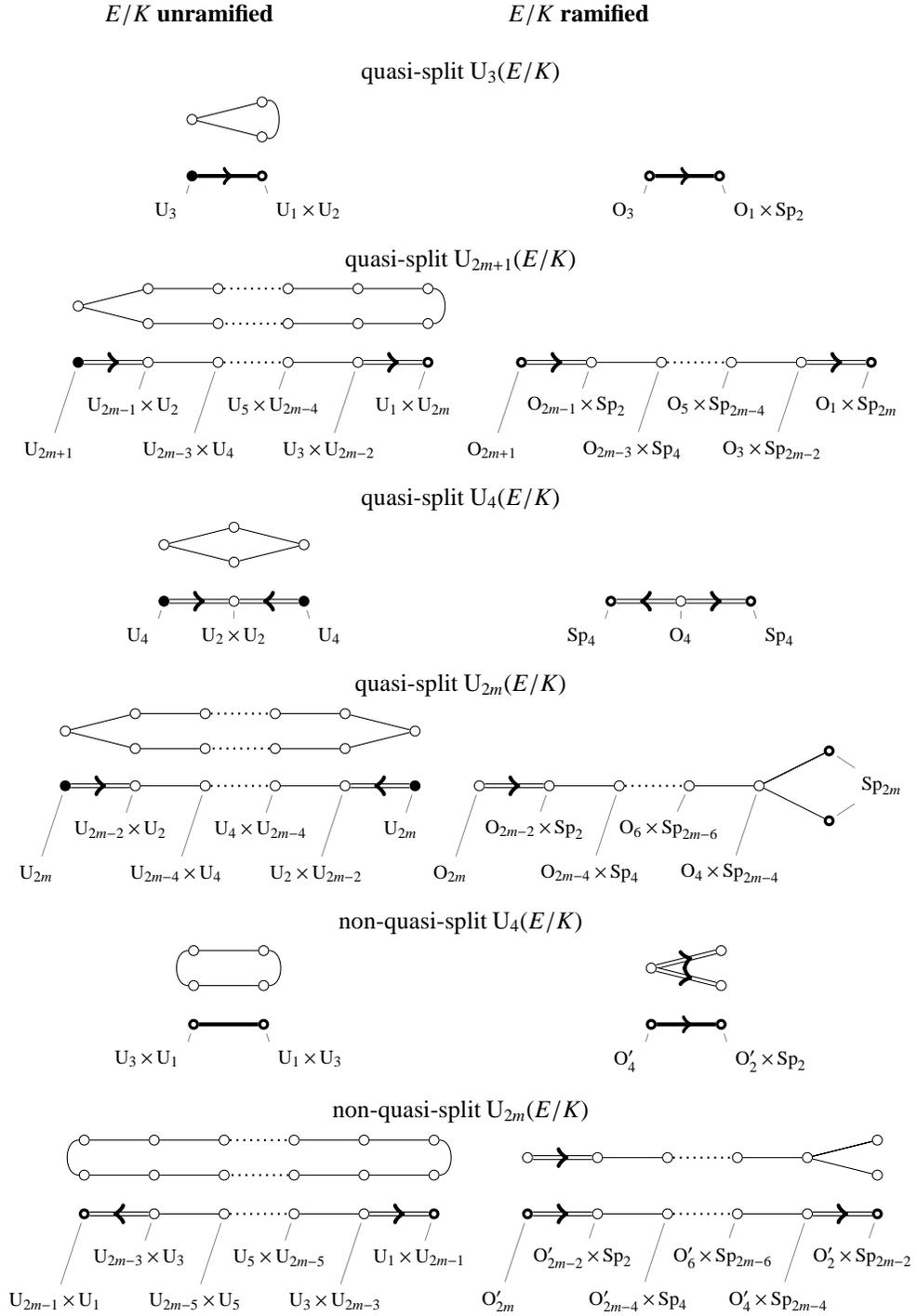}

In addition to the diagrams, Figure \ref{fig:background:Ureductions} also gives the groups $\fbar{x}$ for each vertex $x$ in the closure of the fundamental alcove.  Tits gives a detailed description of how these reductions are derived \cite{tits:79a}*{\S 2.10, \S 3.11} in the case of odd quasi-split unitary groups.  For even groups, the analogous results can be determined using \cite{tits:79a}*{3.5.1}; we will also argue more directly in the proof of Theorem \ref{thm:induction:Ured}.  See also Johnson \cite{johnson:68a} for a discussion of lattices in Hermitian spaces.  We denote by $\Orth_n$ the split orthogonal group over $\k$, and by $\Orth_n'$ the non-split orthogonal group.

\bibliography{bibliography/Biblio}

\begin{bibdiv}
\begin{biblist}

\bib{arthur:EndoscopicClassification}{book}{
      author={Arthur, James},
       title={The endoscopic classification of representations: orthogonal and
  symplectic groups},
      series={Colloquium Publications (61)},
   publisher={American Math Society},
     address={Providence R.I.},
        date={2013},
}

\bib{bosch-lutkebohmert-reynaud:NeronModels}{book}{
      author={Bosch, Siegfried},
      author={L\"utkebohmert, Werner},
      author={Reynaud, Michel},
       title={{N\'eron models}},
   publisher={Springer-Verlag},
     address={Berlin},
        date={1980},
}

\bib{bruhat-tits:72a}{article}{
      author={Bruhat, Francois},
      author={Tits, Jacques},
       title={Groupes r\'eductifs sur un corps local, chapitre i},
        date={1972},
     journal={Publ. Math. I.H.E.S.},
      volume={41},
       pages={5\ndash 251},
}

\bib{bruhat-tits:84a}{article}{
      author={Bruhat, Francois},
      author={Tits, Jacques},
       title={Groupes r\'eductifs sur un corps local, chapitre ii},
        date={1984},
     journal={Publ. Math. I.H.E.S.},
      volume={60},
       pages={197\ndash 376},
}

\bib{bruhat-tits:84b}{article}{
      author={Bruhat, Francois},
      author={Tits, Jacques},
       title={{Sch\'emas en groupes et immeubles des groupes classiques sur un
  corps local}},
        date={1984},
     journal={Bull. Soc. Math. France},
      volume={112},
      number={2},
       pages={259\ndash 301},
}

\bib{bruhat-tits:87b}{article}{
      author={Bruhat, Francois},
      author={Tits, Jacques},
       title={{Groupes alg\'ebriques sur un corps local. Chapitre III.
  Compl\'ements et applications \`a la cohomologie galoisienne}},
        date={1987},
     journal={J. Fac. Sci. Univ. Tokyo Sect. IA Math.},
      volume={34},
      number={3},
       pages={671\ndash 698},
}

\bib{bruhat-tits:87a}{article}{
      author={Bruhat, Francois},
      author={Tits, Jacques},
       title={{Sch\'emas en groupes et immeubles des groupes classiques sur un
  corps local. II}},
        date={1987},
     journal={Bull. Soc. Math. France},
      volume={115},
      number={2},
       pages={141\ndash 195},
}

\bib{carter:72a}{article}{
      author={Carter, Roger~W.},
       title={{Conjugacy classes in the Weyl group}},
        date={1972},
     journal={Compositio Mathematica},
      volume={25},
      number={1},
       pages={1\ndash 59},
}

\bib{carter:93a}{book}{
      author={Carter, Roger~W.},
       title={{Finite groups of Lie type: conjugacy classes and complex
  characters}},
   publisher={Wiley \& Sons},
     address={Chichester},
        date={1993},
}

\bib{reeder-debacker:09a}{article}{
      author={DeBacker, Stephen},
      author={Reeder, Mark},
       title={{Depth-zero supercuspidal L-packets and their stability}},
        date={2009},
     journal={Annals of Math},
      volume={169},
      number={3},
       pages={795\ndash 901},
}

\bib{garrett:BuildingsAndClassicalGroups}{book}{
      author={Garrett, Paul},
       title={Buildings and classical groups},
   publisher={{Chapman \& Hall}},
     address={London},
        date={1997},
}

\bib{gerstein:BasicQuadraticForms}{book}{
      author={Gerstein, Larry~J.},
       title={Basic quadratic forms},
      series={Graduate Studies in Mathematics},
   publisher={American Math Society},
     address={Providence R.I.},
        date={2008},
}

\bib{gross:04a}{inproceedings}{
      author={Gross, Benedict~H.},
       title={Heegner points and representation theory},
        date={2004},
   booktitle={{Heegner points and Rankin L-series}},
      editor={Darmon, Henri},
      editor={Zhang, Shou-Wu},
      series={MSRI Publications},
      volume={49},
   publisher={Cambridge UP},
     address={Cambridge, UK},
       pages={37\ndash 65},
}

\bib{gross-reeder:09a}{article}{
      author={Gross, Benedict~H.},
      author={Reeder, Mark},
       title={{Arithmetic invariants of discrete Langlands parameters}},
        date={2010},
     journal={Duke Math Journal},
       pages={431\ndash 508},
}

\bib{harris-taylor:01a}{book}{
      author={Harris, Michael},
      author={Taylor, Richard},
       title={{The geometry and cohomology of some simple Shimura varieties}},
   publisher={Princeton U.P.},
     address={Princeton},
        date={2001},
}

\bib{henniart:00a}{article}{
      author={Henniart, Guy},
       title={{Une preuve simple des conjectures de Langlands pour $\GL(n)$ sur
  un corps $p$-adique}},
        date={2000},
     journal={Invent. Math.},
      volume={139},
      number={2},
       pages={439\ndash 455},
}

\bib{humphreys:95a}{book}{
      author={Humphreys, James~E.},
       title={Conjugacy classes in semisimple algebraic groups},
   publisher={American Math Society},
     address={Providence R.I.},
        date={1995},
}

\bib{isaacs:CharacterTheoryOfFiniteGroups}{book}{
      author={Isaacs, I.~Martin},
       title={Character theory of finite groups},
   publisher={Academic Press},
     address={New York},
        date={1976},
}

\bib{johnson:68a}{article}{
      author={Johnson, Arnold},
       title={Integral representations of hermitian forms over local fields},
        date={1968},
     journal={Journ. reine u. angewandte},
      volume={229},
       pages={57\ndash 80},
}

\bib{lang:56a}{article}{
      author={Lang, Serge},
       title={Algebraic groups over finite fields},
        date={1956},
     journal={Amer. J. Math.},
      volume={78},
       pages={555\ndash 563},
}

\bib{moy-prasad:96a}{article}{
      author={Moy, Allen},
      author={Prasad, Gopal},
       title={{Jacquet functors and unrefined minimal $K$-types}},
        date={1996},
     journal={Comment. Math. Helv.},
      volume={71},
      number={1},
       pages={98\ndash 121},
}

\bib{prasad-yu:02a}{article}{
      author={Prasad, Gopal},
      author={Yu, Jiu-Kang},
       title={On finite group actions on reductive groups and buildings},
        date={2002},
     journal={Inventiones Math.},
      volume={147},
       pages={545\ndash 560},
}

\bib{reeder:10a}{article}{
      author={Reeder, Mark},
       title={{Torsion automorphisms of simple Lie algebras}},
        date={2009},
     journal={{L'Enseignement Mathematique}},
      volume={56},
      number={2},
       pages={{3\ndash 47}},
}

\bib{reeder:11a}{article}{
      author={Reeder, Mark},
       title={{Elliptic centralizers in Weyl groups and their coinvariant
  representations}},
        date={2011},
     journal={Representation Theory},
      volume={15},
       pages={{63\ndash 111}},
}

\bib{roe:11a}{thesis}{
      author={Roe, David},
       title={The local langlands correspondence for tamely ramified groups},
      school={Harvard University},
        year={2011},
}

\bib{serre:LocalClassFieldThy}{incollection}{
      author={Serre, Jean-Pierre},
       title={Local class field theory},
        date={1967},
   booktitle={{Algebraic number theory}},
      editor={Cassels, John W.~S.},
      editor={Frohlich, Albrecht},
   publisher={Academic Press},
     address={London},
       pages={129\ndash 162},
}

\bib{serre:LocalFields}{book}{
      author={Serre, Jean-Pierre},
       title={Local fields},
   publisher={Springer-Verlag},
     address={New York},
        date={1979},
}

\bib{serre:GaloisCohomology}{book}{
      author={Serre, Jean-Pierre},
       title={Galois cohomology},
     edition={Second},
   publisher={Springer-Verlag},
     address={Berlin},
        date={2001},
}

\bib{springer:LinearAlgebraicGroups}{book}{
      author={Springer, T.A.},
       title={Linear algebraic groups},
     edition={2nd ed.},
   publisher={Birkh\"auser},
     address={Boston},
        date={1998},
}

\bib{tits:79a}{inproceedings}{
      author={Tits, Jacques},
       title={Reductive groups over local fields},
        date={1979},
   booktitle={{Automorphic forms, representations and L-functions (Proc.
  Sympos. Pure Math., Oregon State Univ., Corvallis OR, 1977)}},
      series={Proceedings of Symposia in Pure Mathematics},
   publisher={American Math Society},
     address={Providence R.I.},
       pages={29\ndash 69},
}

\bib{vogan:93a}{article}{
      author={Vogan, David},
       title={{The local Langlands conjecture}},
        date={1993},
     journal={Representation Theory of Groups and Algebras},
      volume={145},
}

\bib{yu:03a}{unpublished}{
      author={Yu, Jiu-Kang},
       title={{Smooth models associated to concave functions in Bruhat-Tits
  theory}},
        date={2003},
        note={Preprint},
}

\bib{yu:09b}{incollection}{
      author={Yu, Jiu-Kang},
       title={{Bruhat-Tits theory and buildings}},
        date={2009},
   booktitle={Ottowa lectures on admissible representations of reductive
  $p$-adic groups},
      editor={Cunningham, Clifton},
      editor={Nevins, Monica},
      series={Fields Institute Monographs},
   publisher={American Math Society},
     address={Providence R.I.},
       pages={53\ndash 76},
}

\bib{yu:09a}{incollection}{
      author={Yu, Jiu-Kang},
       title={{On the local Langlands correspondence for tori}},
        date={2009},
   booktitle={Ottawa lectures on admissible representations of reductive
  $p$-adic groups},
      editor={Cunningham, Clifton},
      editor={Nevins, Monica},
      series={Fields Institute Monographs},
      volume={26},
   publisher={American Math Society},
     address={Providence R.I.},
       pages={177\ndash 183},
}

\end{biblist}
\end{bibdiv}

\end{document}